\documentclass[12pt,reqno]{amsart}

\usepackage[shortlabels]{enumitem}
\usepackage{esint}
\usepackage{physics}
\usepackage{relsize}
\usepackage[backref]{hyperref}
\usepackage{mathtools}
\usepackage{amsfonts}
\usepackage[all]{xy}
\usepackage{geometry}
\newtheorem{theorem}{Theorem}[section]
\newtheorem{lemma}[theorem]{Lemma}
\newtheorem{definition}[theorem]{Definition}
\newtheorem{example}[theorem]{Example}
\theoremstyle{corollary}
\newtheorem{corollary}[theorem]{Corollary}
\theoremstyle{conjecture}
\newtheorem{conjecture}{Conjecture}

\newtheorem*{claim}{Claim}
\theoremstyle{assumption}
\newtheorem{assumption}{Assumption}

\theoremstyle{proposition}
\newtheorem{proposition}[theorem]{Proposition}
\theoremstyle{remark}
\newtheorem{remark}[theorem]{Remark}
\numberwithin{equation}{section}

\newcommand{\olsi}[1]{\,\overline{\!{#1}}} % olsi short italic
\newcommand{\ols}[1]{\mskip.5\thinmuskip\overline{\mskip-.5\thinmuskip {#1} \mskip-.5\thinmuskip}\mskip.5\thinmuskip} % overline short
\newcommand{\ii}{\ensuremath{\sqrt{-1}}}
\newcommand{\pp}{\bar\partial}
\everymath{\displaystyle}
 \date{\today}

\begin{document}

\title[\resizebox{5.5in}{!}{Hermitian-Yang-Mills connections on some complete non-compact K\"ahler manifolds}]{Hermitian-Yang-Mills connections on some complete non-compact K\"ahler manifolds}
\author{Junsheng Zhang}
\address{Department of Mathematics\\
 University of California\\
 Berkeley, CA, USA, 94720\\}
\curraddr{}
\email{jszhang@berkeley.edu}
\thanks{}

%\subjclass[2010]{Primary }

\begin{abstract}We give an algebraic criterion for the existence of projectively Hermitian-Yang-Mills metrics on a holomorphic vector bundle $E$ over some complete non-compact  K\"ahler manifolds $(X,\omega)$, where $X$ is the complement of a divisor in a compact K\"ahler manifold and we impose some conditions on the cohomology class and the asymptotic behaviour of the K\"ahler form $\omega$. We introduce the notion of stability with respect to a pair of $(1,1)$-classes which generalizes the standard slope stability. We prove that this new stability condition is both sufficient and necessary for the existence of projectively Hermitian-Yang-Mills metrics in our setting.
 \end{abstract}

\maketitle
\section{Introduction}

The celebrated Donaldson-Uhlenbeck-Yau theorem \cite{donaldson1985,uhlenbeck-yau} says that on a compact K\"ahler manifold $(X,\omega)$, an irreducible holomorphic vector bundle $E$ admits a Hermitian-Yang-Mills (HYM) metric if and only if it is $\omega$-stable. After this pioneering work, there have been several results aiming to generalize this to non-compact K\"ahler manifolds \cite{simpson,bando,ni-ren,ni,jacob-walpuski,mochizuki}. A key issue is to understand what role  stability plays on the existence of projectively Hermitian-Yang-Mills (PHYM) metrics. An interesting special case in the non-compact setting is when $(X,E)$ both can be compactified, i.e. $X$ is the complement of a divisor in a compact K\"ahler manifold $\olsi{X}$ and $E$ is the restriction of a holomorphic vector bundle $\olsi{E}$ on $\olsi{X}$, and when the K\"ahler metric has a known asymptotic behaviour. Under these assumptions, one wants to build a relation between the existence of PHYM metrics on $E$ and some algebraic data on $\olsi E$. In this paper, we prove a result in this setting.

Let $\olsi{X}$ be an $n$-dimensional ($n\geq 2$) compact K\"ahler  manifold, $D$ be a smooth divisor and $X=\olsi{X}\backslash D$ denote the complement of $D$ in $\olsi{X}$. Let $\olsi{E}$ be a holomorphic vector bundle on $\olsi{X}$, which we always assume to be irreducible unless otherwise mentioned. Let $E$, $\olsi{E}|_D$ denote its restriction to $X$ and $D$ respectively. 
Suppose the normal bundle $N_D$ of $D$ in $\olsi X$ is ample. On $X$ we consider complete K\"ahler metrics $\omega=\omega_0+dd^c\varphi$ satisfying \textit{Assumption 1} (see Section \ref{assumption on the asymptotic behaviour} for a precise definition). Roughly speaking, we assume $\omega_0$ is a smooth closed (1,1)-form on $\olsi{X}$ vanishing when restricted to $D$, $\varphi$ is a smooth function on $X$, and $\omega$ is asymptotic to some model K\"ahler metrics given explicitly on the punctured disc bundle of $N_D$. Typical examples satisfying these assumptions are Calabi-Yau metrics on the complement of an anticanonical divisor of a Fano manifold and its generalizations \cite{tian-yau1,hsvz1,hsvz2} (see Section \ref{examples} for a sketch).

To state our theorem, we need two ingredients:  the existence of a good initial hermitian metric on $E$ and the definition for stability with respect to a pair of classes. The following lemma is proved in Section \ref{existence of good initial metric}.
\begin{lemma}\label{good initial metric}
	If $\olsi{E}|_{D}$ is $c_1(N_D)$-polystable, then there is a hermitian metric $H_0$ on $E$ satisfying:
	\begin{itemize}
		\item [(1).] there is a hermitian metric $\olsi{H}_0$ on $\olsi{E}$ and a function $f\in C^{\infty}(X)$ such that $H_0={e}^f\olsi{H}_0$,
	\item [(2).] $|\Lambda_{\omega}F_{H_0}|=O(r^{-N_0})$, where $r$ denotes the distance function to a fixed point induced by the metric $w$ and $N_0$ is the number in \textit{Assumption 1}-(3).
	\end{itemize}
\end{lemma} 

 We call $H_0$ conformal to a smooth extendable metric if it satisfies the first condition in Lemma \ref{good initial metric}. A key feature we use in this paper is that the induced metic on $\operatorname {End}(E)$ is conformally invariant with respect to metrics on $E$. Therefore the two hermitian metrics $H_0$ and $\olsi{H}_0$ induce the same metric on $\operatorname{End}(E)$ and this is the norm used in Lemma \ref{good initial metric}-(2). Then naturally (following \cite{simpson}) one wants to find a PHYM metric in the following set
\begin{equation}\label{define of P_H}
	\mathcal P_{H_0}=\left\{H_0 e^s: s\in C^{\infty}(X,\ii \mathfrak{su}(E,H_0)), \left\Vert s\right \Vert_{L^{\infty}}+\left\Vert \olsi{\partial}s\right\Vert_{L^{2}}< \infty\right\}.
\end{equation}
Here we use $\ii\mathfrak {su}(E,H_0)$ to denote the subbundle of $\operatorname{End}(E)$ consisting of the trace-free and self-adjoint endomorphisms with respect to $H_0$. Though $H_0$ in general is not unique, we will show that if we fix the induced metric on $\det E$, then the set $\mathcal P_{H_0}$ is uniquely determined  as long as $H_0$ satisfies conditions in Lemma \ref{good initial metric} (see Proposition \ref{uniqueness of P_H}).

Next we define stability with respect to a pair of (1,1)-classes, which generalizes the standard slope stability defined for K\"ahler classes in  \cite[Chapter 5]{kobayashi-book}. In the following, we use $\mu_{\alpha}(S)$ to denote the slope of a torsion-free coherent sheaf with respect to a class $\alpha\in H^{1,1}(M)$ on a compact K\"ahler manifold $M$ (see Section \ref{definition of stability with respect to a pair} for a more detailed discussion), i.e. 
\begin{equation*}
	\mu_{\alpha}(S):=\frac{1}{\rank (S)}\int_{M} c_1(\det S)\wedge \alpha^{n-1}.
\end{equation*}

\begin{definition}\label{definition of alpha beta stability}
	Let $M$ be a compact K\"ahler manifold, $\alpha,\beta\in H^{1,1}(M)$ be two classes, $E$ be a holomorphic vector bundle over $M$.
 We say $E$ is $(\alpha,\beta)$-stable if every coherent reflexive subsheaf $S$ of $E$ with $0<\rank (S)<\rank (E)$ satisfies either of the following conditions:
	\begin{itemize}
		\item [(a)] $\mu_{\alpha}(S)< \mu_{\alpha}(E)$, or 
		\item [(b)] $\mu_{\alpha}(S)= \mu_{\alpha}(E)$ and $\mu_{\beta}(S)<\mu_{\beta}(E)$.
	\end{itemize}
\end{definition}

The main result of this paper is
\begin{theorem}\label{main theorem}
	Let $\omega=\omega_0+dd^c\varphi$ be a K\"ahler metric satisfying the \textit{Assumption 1} in Section \ref{assumption on the asymptotic behaviour}. Suppose $\olsi{E}|_{D}$  is $c_1(N_D)$-polystable and $\mathcal P_{H_0}$ is defined by (\ref{define of P_H}). Then there exists a unique $\omega$-PHYM metric in $\mathcal P_{H_0}$ if and only if $\olsi{E}$ is $\left(c_1(D),[\omega_0]\right)$-stable.
\end{theorem}

 By the definition of $(\alpha,\beta)$-stability in Definition \ref{definition of alpha beta stability}, we have the following consequence.
\begin{corollary} Suppose $\omega$ and $\olsi{E}$ satisfy the conditions in Theorem \ref{main theorem}. Then we have 
\begin{itemize}
	\item [(1).] suppose $[\omega_0]=0$, then there exists a unique $\omega$-PHYM metric in $\mathcal P_{H_0}$ if and only if $\olsi{E}$ is $c_1(D)$-stable.
	\item [(2).] if $\olsi{E}|_{D}$ is $c_1(N_D)$-stable, then there exists a unique $\omega$-PHYM in $\mathcal P_{H_0}$.
\end{itemize}
	\end{corollary}
	
Now let us give a brief outline for the proof of Theorem \ref{main theorem}. For the ``if'' direction, we follow the argument in  \cite{simpson,mochizuki} by solving Dirichlet problems on a sequence of domains exhausting $X$. A key issue here is to prove a uniform $C^0$-estimate. For this we rely on a weighted Sobolev inequality in \cite[Proposition 2.1]{tian-yau1} and Lemma \ref{key observation} which builds a relation between weakly holomorphic projection maps over $X$ and coherent  subsheaves over $\olsi X$. For the ``only if'' direction, we use integration by parts to show that the curvature form on $E$ can be used to compute the degree of $\olsi E$ with respect to $[\omega_0]$ (see Lemma \ref{degree coincide}). For both directions, the asymptotic behaviour of the K\"ahler metric $\omega$ plays an essential role.

Then let us compare Theorem \ref{main theorem} with some results existing in the literature. In \cite{simpson} and \cite{mochizuki}, by assuming some conditions on the base K\"ahler manifold $(X,\omega)$ and an initial hermitian metric on $E$, it was proved  that for an irreducible vector bundle $E$ the existence of a PHYM metric is equivalent to a stability condition called analytic stability. In our case, since we assume that $E$ has a compactification $\olsi{E}$ on $\olsi{X}$, the existence of good initial metrics is guaranteed by the polystablity assumption of $\olsi{E}|_D$. Moreover the stability we used in Theorem \ref{main theorem} is for $\olsi{E}$ which is purely algebraic, i.e. independent of choice of metrics. In \cite{bando}, for asymptotically conical K\"ahler metrics on $X$, it was proved that if $\olsi{E}|_D$ is $c_1(N_D)$-polystable, then there exists PHYM metrics on $E$.  No extra stability condition is needed in this case. Therefore the necessity of stability conditions depends on the geometry of $(X,\omega)$ at infinity. 
%(More precisely, \cite[Theorem 4]{bando} says that if $\olsi{E}|_D$ is $c_1(N_D)$-polystable of degree 0, then there exists HYM metrics on $E$. But the same proof gives the statement we quote here. See also Section \ref{relation with some results}.)
%Moreover we do have examples which do not satisfy the stability condition in Theorem \ref{main theorem} (see Example \ref{example of nonsplitting}). %As mentioned in the introduction of \cite{mochizuki}, the efficiency of the stability condition depends on the geometry of $(X,\omega)$ at infinity.
 Another typical example for such a phenomenon is the problem for the existence of bounded solutions of the Poisson equation on noncompact manifolds. See Section \ref{relation with some results} for a brief discussion.

The paper is organized as follows. In Section \ref{assumption on the asymptotic behaviour}, we discuss the assumptions on the K\"ahler manifold $(X,\omega)$ and prove a weighted mean value inequality for nonnegative almost subharmonic functions. In Section \ref{basics on vector bundle}, we give a brief review of some standard results for hermitian holomorphic vector bundles and give a detailed discussion on $(\alpha,\beta)$-stability used in Theorem \ref{main theorem}. In Section \ref{existence of good initial metric}, Lemma \ref{good initial metric} is proved and we also show that the assumption in Lemma \ref{good initial metric} is necessary.
In Section \ref{proof of the main theorem}, we prove Theorem \ref{main theorem} and give an example which does not satisfy the stability assumption. In Section \ref{relation with some results}, we discuss some other results on the existence of PHYM metrics. In Section \ref{examples}, we discuss some Calabi-Yau metrics satisfying \textit{Assumption 1}. In Section \ref{general normal bundles}, we prove a counterpart of Theorem \ref{main theorem} in a different setting where $\olsi{X}$ is a compact K\"ahler surface and $c_1(N_D)$ is trivial. In Section \ref{open problems}, we discuss some problems for further study.

\subsection*{Notations and conventions}
\begin{itemize}
	\item $d^c=\frac{\ii}{2}(-\partial+\pp)$, so $dd^c=\ii \partial\pp$.
	\item $\Delta=\ii \Lambda\pp\partial$, so in local normal coordinates $\Delta f=-\sum_{i=1}^{n}\frac{\partial^2f}{\partial z_i\partial \bar{z_i}}$.
	\item $B_r(p)$ denotes the geodesic ball centered at $p$ with radius $r$ and if the basepoint $p$ is clear from the context, we will just write it as $B_r$.
\item In this paper, we identify a holomorphic vector bundle with the sheaf formed by its holomorphic sections.
\item When we integrate on a Riemannian manifold $(M,g)$, typically we will omit the volume element $d V_g$.
	\item Let $(M,\omega)$ be a K\"ahler manifold and $(E,H)$ be a hermitian holomorphic vector bundle over $M$. We use $C^{\infty}(M,E)$ to denote smooth sections of $E$; 
\begin{equation*}
	\text{$W^{k,p}(M,E;\omega,H)$ (respectively $W_{loc}^{k,p}(M,E;\omega,H)$)}
\end{equation*} to denote sections of $E$ which are $W^{k,p}$ (respectively $W_{loc}^{k,p}$) with respect to the metric $\omega$ and $H$. If bundles or metrics are clear from the text, we will omit them.
\end{itemize}
\subsection*{Acknowledgements} The author wants to express his deep and sincere gratitude to his advisor Song Sun, for the patient guidance, inspiring discussions and valuable suggestions. He thanks Charles Cifarelli for careful reading of the draft and numerous suggestions. He thanks Sean Paul for his interest in this work. The author is partially supported by NSF grant DMS-1810700 and NSF grant DMS-2004261.
%%%%%%%%%%%%%%%%%%%%%%%%%%%%%%%%%%%%%%%%%%%%%%%%%%%%%%%
%%%%%%%%%%%%%%%%%%%%%%%%%%%%%%%%%%%%%%%%%%%%%%%%%%%%%%%
%%%%%%%%%%%%%%%%%%%%%%%%%%%%%%%%%%%%%%%%%%%%%%%%%%%%%%%

\section{On the asymptotic behaviour of $\omega$}\label{assumption on the asymptotic behaviour} As mentioned in the introduction, the asymptotic behaviour of the K\"ahler metric on the base manifold is crucial to make the argument in this paper work. In this section we will discuss these assumptions. 
Let $\olsi{X}$ be an $n$-dimensional ($n\geq 2$) compact K\"ahler manifold, $D$ be a smooth divisor and $X=\olsi{X}\backslash D$ denote the complement of $D$ in $\olsi{X}$.
Let $L_D$ denote the holomorphic line bundle determined by $D$.

 From now on, we assume the normal bundle of $D$, i.e. $N_D=L_D|_D$ is ample unless otherwise mentioned. Then we know that $c_1(D)$ is nef and big. We fix a hermitian metric $h_D$ on $N_D$ such that $$\omega_D:=\ii \Theta_{h_D}$$ is a K\"ahler form on $D$, where $\Theta_{h_D}$ denote the curvature of $h_D$.  Let $\mathcal D$ (respectively $\mathcal C$) denote the (respectively punctured) disc bundle of $N_D$ under the metric $h_D$, i.e. 
\begin{equation}\label{definition of model space}
\begin{aligned}
	\mathcal C&:=\left\{\xi\in N_D:0<\left|\xi\right|_{h_D}\leq1/2\right\},\\
	\mathcal D&:=\left\{\xi\in N_D:\left|\xi\right|_{h_D}\leq1\right\}.
\end{aligned}
	\end{equation}
We are mainly interested in the region where $|\xi|_{h_D}$ is small, which will be viewed as a model of $X$ at infinity. Then we have a well-defined positive smooth function on $\mathcal C$ 
\begin{equation}\label {definition of t}
	t=-\log|\xi|^2_{h_D}
\end{equation}
 such that $\ii\partial\pp t=\omega_D$, where using the obvious projection map $p:\mathcal C\rightarrow D$, we view $\omega_D$ as a form on $\mathcal C$.
Then for every smooth function $F:(0,\infty)\longrightarrow \mathbb R$ with $F'>0$ and $F''>0$,
\begin{equation}\label{explicit expression of kahler form}
\omega_{ F}=\ii\partial \pp F(t)=F'(t)\ii\partial\pp t+F''(t)\ii\partial t\wedge \pp t
\end{equation}
defines a K\"ahler form on $\mathcal C$. Let $g_{F}$ denote the corresponding Riemannian metric and $r_{F}$ denote the induced distance function to a fixed point $p$. 

We need a diffeomorphism to identify a neighborhood of $D$ in $N_D$ with a neighborhood of $D$ in $\olsi{X}$. For this, we use the following definition introduced by Conlon and Hein in \cite[Definition 4.5]{conlon-hein2}.

\begin{definition}\label{definition of expoenntial map}
	An exponential-type map is a diffeomorphism $\Phi$ from a neighborhood of $D$ in $N_D$ to a neighborhood of $D$ in $\olsi{X}$ such that 
	\begin{itemize}
		\item [(1)]	$\Phi(p)=p$ for all $p\in D$,
		\item [(2)] $d\Phi_p$ is complex linear for all $p\in D$,
		\item [(3)] $\pi(d\Phi_p(v))=v$ for all $p\in D$ and $v\in N_{D,p}\subset T^{1,0}_pN_D$, where $\pi$ denotes the projection $T^{1,0}_p\olsi{X}\rightarrow T^{1,0}_p\olsi{X}/T^{1,0}_pD=N_{D,p}$.
	\end{itemize}
\end{definition}

%Since we assumed $\pp \Phi=0$ on $D$, we obtain that 
 %\begin{equation}\label{complex structure close}
 	%\text{$\Phi^*J_{\olsi{X}}-J_{\mathcal D}$ is smooth section of $\operatorname{End}(T\mathcal D)$ vanishing on $D$}.
 	%\begin{equation}\label{complex structure is close}
 		%\left|\nabla^k_{g_F}(\Phi^*J_{\olsi{X}}-J_{\mathcal D})\right|_{g_F}=O(e^{-\delta t}) \text{ for $k=0,1$ and some $\delta>0$} \text{ and}
 	%\end{equation}
% \end{equation}
Now we can state the assumptions for the K\"ahler metric $\omega$ on $X$. We consider a special class of potentials: 
\begin{equation}\label{conditions on potentials}
	\mathcal H:=\left\{F(t):F(t)=A t^a \text{ for some constant $A>0$ and $a\in \Big(1,\frac{n}{n-1}\Big]$}\right\}
\end{equation}

\begin{assumption}
	\label{assumption on kahler metrics}
Let $\omega$ be a K\"ahler form on $X$ and $g$ be the corresponding Riemannian metric. We assume that  
\begin{itemize}
	\item [(1)] the sectional curvature of $g$ is bounded,
	\item [(2)] $\omega$ can be written as $\omega_0+\ii\partial\pp \varphi$, where $\varphi$ a smooth function on $X$ and $\omega_0$ is a smooth $(1,1)$-form on $\olsi{X}$ with $\omega_0|_D=0$,
	\item [(3)] there exists an exponential-type map $\Phi$ from a neighborhood of $D$ in $N_D$ to a neighborhood of $D$ in $\olsi{X}$ and a potential $F\in \mathcal H$
such that
\begin{equation}\label{assumption on asymptotics}
	\left|\Phi^*(\omega)-\omega_{F}\right|_{g_{F}}=O(r_F^{-N_0}) \text{ for some number  $N_0\geq 8$.}
\end{equation}
\end{itemize}
\end{assumption}

\begin{remark}
	There are (lots of) other potentials $F$ besides those given in (\ref{conditions on potentials}) making the argument in this paper work, but for simplicity of the statement and some computations we only consider potentials in $\mathcal H$. The order in (\ref{assumption on asymptotics}) is not optimal either and again we just choose the number 8 for a neat statement. From now on, unless otherwise mentioned, $N_0$ denote the number in (\ref{assumption on asymptotics}).
	 \end{remark}

 Here are the main properties we will use for K\"ahler metrics defined by the potentials in $\mathcal H$. For simplicity of notation, we omit the subscript for the dependence on $F$.

\begin{proposition}\label{basic geometric property} For the K\"ahler metric defined by a potential $F=At^a\in \mathcal H$, we have 
\begin{itemize}
	\item [(1)] The metric is complete as $|\xi|_{h_D}\rightarrow 0$, 
	\item [(2)] $r\sim t^{\frac{a}{2}}$,
	\item [(3)] $\operatorname {Vol}(B_r(p))\sim r^{2n(1-\frac{1}{a})}$,
	\item [(4)] if $\theta$ is a smooth form on $\mathcal D$ with $\theta|_D=0$, then $\left|\theta\right|_g=O(e^{-\delta t})$ for some $\delta>0$.\end{itemize}

\end{proposition}
Conditions (1)--(3) follow directly from (\ref{explicit expression of kahler form})  and (\ref{conditions on potentials}). Condition (4) can be proved directly by doing computation in local coordinates on $\mathcal D$ as in \cite[Section 3]{hsvz1}. For completeness and later reference, we include some details.

\textit{Proof of (4)}:
 We choose local holomorphic coordinates $\underline{z}=\{z_i\}_{i=1}^{n-1}$ on the smooth divisor $D$ and fix a local holomorphic trivialization $e_0$ of $N_D$ with $|e_0|_{h_D}=e^{-\psi}$, where $\psi$ is a smooth function on $D$ satisfying $\ii \partial\pp \psi=\omega_D$. Then we get local holomorphic coordinates $\{z_1,\cdots,z_{n-1},w\}$ on $\mathcal C$ by writing a point $\xi=we_0(\underline z)$. Then in these coordinates we can write (\ref{explicit expression of kahler form}) as 
\begin{equation}\label{model metric expansion}
\omega_{F}= \ii
F'(t) \psi_{i\bar j} dz_i\wedge d\bar z_{ j}+F''(t) \sqrt{-1}\left(\frac{d w}{w}- \psi_idz_i\right) \wedge\left(\frac{d \bar{w}}{\bar{w}}-\psi_{\bar j}d\bar z_{ j}\right).
\end{equation}
Then it is easy to check the following estimates:
	\begin{equation}\label{estimate for 1-forms}
	\begin{aligned}
			&|dz_i|^2_{\omega_{F}}=|\Lambda_{\omega_F}(dz_i\wedge d\bar z_i)| \sim  t^{1-a}\\
			&|dw|^2_{\omega_{F}}\sim \frac{|w|^2}{F''(t)}\leq Ce^{-\delta t} \text{ for some $\delta>0$}\\
			&\omega_{F}^n \sim \frac{F'(t)^{n-1}F''(t)}{|\omega|^2}\ii^n\left(\prod_{i=1}^{n-1}dz_i\wedge d\bar z_i\right)\wedge dw\wedge d\bar w .
	\end{aligned}
	\end{equation}	
 Then (4) follows directly from (\ref{estimate for 1-forms}).
\qed

\begin{remark}
	Actually, from the proof of Proposition \ref{basic geometric property}-(4), we can give an effective lower bound for $\delta$. For example, for 2-forms, $\delta$ can be chosen to be any positive number sufficiently close to (and less than) $1/2$. However $\delta>0$ is sufficient for our later use.
\end{remark}

\begin{remark}
	Although not needed in this paper, we mention that following the computation in \cite[Section 4]{tian-yau1} or \cite[Section 3]{biquard-guenancia}, we can show that $\Vert Rm\Vert\leq C r^{2(\frac{1}{a}-1)}$.

\end{remark}

In \textit{Assumption 1}, we only assume the asymptotics of the K\"ahler forms. To get the asymptotic behaviour of the corresponding Riemannian  metrics, we need to show that the complex structure of $\olsi{X}$ and $\mathcal D$ are sufficiently close under the metric $g_F$. When $D$ is an anticanonical divisor, the following result is proved in \cite[Proposition 3.4]{hsvz1}. For a general smooth divisor $D$, the author learned the following proof from Song Sun.

\begin{lemma}
	Let $J_{\mathcal D}$ and $J_{\ols{X}}$ denote the complex structure on $\mathcal D$ and $\olsi{X}$ respectively. And $\Phi^*J_{\olsi{X}}:=d\Phi\circ J_{\olsi{X}}\circ (d\Phi)^{-1}$ denote the pullback of $J_{\olsi{X}}$ under an exponential-type map $\Phi$. Then we have \begin{equation}\label{complex structure is close}
 		\left|\nabla^k_{g_F}(\Phi^*J_{\olsi{X}}-J_{\mathcal D})\right|_{g_F}=O(e^{(-\frac{1}{2}+\epsilon) t}) \text{ for all $k\geq 0$ and $\epsilon>0$}.
 	\end{equation}
\end{lemma}

\begin{proof}
 	Since $d\Phi_p$ is complex linear for all $p\in D$, we know $\Phi^*J_{\olsi{X}}-J_{\mathcal D}$ is smooth section of $\operatorname{End}(T\mathcal D)$ vanishing on $D$. But this is not enough to get the bound  claimed in (\ref{complex structure is close}). We will use the integrability of $\Phi^*J_{\olsi{X}}$ and property (3) in Definition \ref{definition of expoenntial map}. In the following, we ignore the pull-back notation.
 	
 	Around a fixed point in $D$ we can choose local holomorphic coordinates $\left\{w,z_1 \cdots, z_{n-1}\right\}$ of the total space of $N_{D}$ so that the zero section is given by $w=0$. Then we can write for $\alpha=1, \cdots, n-1$ that
\begin{equation*}
	J_{\olsi{X}} \partial_{z_{\alpha}}=\ii \partial_{z_{\alpha}}+P_{\alpha} \partial_{\bar{w}}+Q_{\alpha \beta} \partial_{\bar{z}_{\beta}}+O\left(\left|w\right|^{2}\right),
\end{equation*}
where $P_{\alpha}$ and $Q_{\alpha \beta}$ are linear functions of $w$ and $\olsi{w}$, i.e. there are smooth functions $p_{\alpha}$ and $p_{\olsi{\alpha}}$ of $\{z_i\}$ such that $P_{\alpha}=p_{\alpha}w+p_{\ols{\alpha}}\ols{w}$ and a similar expression for $Q_{\alpha\beta}$. There are no type $(1,0)$ vectors in the linear term of the right hand side because $J_{\olsi{X}}^2=-\operatorname{id}$. Since $J_{\olsi{X}}$ is integrable, we know that
\begin{equation*}
	\left[\partial_{w}-\ii J_{\olsi{X}} \partial_{w}, \partial_{z_{\alpha}}-\ii J_{\olsi{X}} \partial_{z_{\alpha}}\right]=-2 \ii \partial_{w} P_{\alpha} \partial_{\bar{w}}-2 \ii \partial_{w} Q_{\alpha \beta} \partial_{\bar{z}_{\beta}}+O\left(\left|w\right|\right)
\end{equation*}
is still of type $(1,0)$ with respect to $J_{\olsi{X}}$, which coincides with $J_{\mathcal D}$ when restricted to $D$. Therefore $$\partial_{w} P_{\alpha}=p_{\alpha}=0.$$

By the property (2) and (3) in Definition \ref{definition of expoenntial map} and the following standard exact sequence of the holomorphic vector bundles on $D$ 

\begin{equation*}
	0\longrightarrow T^{1,0}D\longrightarrow T^{1,0}\olsi{X}\longrightarrow N_D\longrightarrow 0,
\end{equation*}
we know that on $D$, the $d\olsi{z}_{\alpha}$ component of $\pp_{J_{\olsi{X}}}\partial_{w}$ is tangential to $D$. Note that by definition we have $\pp_{J_{\olsi{X}}}\partial_{w}=\mathcal L_{\partial_{w}}J_{\olsi{X}}$, therefore we know that
\begin{equation*}
	\pp_{J_{\olsi{X}}}\partial_{w}(\partial_{\olsi{z}_{\alpha}})=[\partial_{w},J_{\olsi{X}}\partial_{\olsi{z}_{\alpha}}]=\bar{p}_{\olsi{\alpha}} \partial_{w}+\partial_{w}\olsi{Q}_{\alpha \beta}\partial_{z_{\beta}} +O\left(\left|w\right|\right).
\end{equation*}
Since on $D$, the $d\olsi{z}_{\alpha}$ component of $\pp_{J_{\olsi{X}}}\partial_{w}$ is tangential to $D$, we obtain $p_{\olsi{\alpha}}=0$. So we have for $\alpha=1,\cdots, n-1$
\begin{equation}\label{difference of complex structure}
		J_{\olsi{X}} \partial_{z_{\alpha}}=\ii \partial_{z_{\alpha}}+Q_{\alpha \beta} \partial_{\bar{z}_{\beta}}+O\left(\left|w\right|^{2}\right),
\end{equation}
Now on $\mathcal{D}$ we consider the local basis of holomorphic vector fields (with respect to $J_{\mathcal{C}}$):
$$
e_{n}=w \partial_{w}, e_{\alpha}=\partial_{z_{\alpha}}, \alpha=1, \cdots, n-1
$$
and correspondingly $\bar{e}_{n}, \bar{e}_{\alpha}$ the conjugate vector fields, and $e^{n}, e^{\alpha}$ the dual frame etc. Then we can write
\begin{equation}\label{difference after 2.6}
	J_{\olsi{X}}-J_{\mathcal{D}}=\sum J_{i}^{j} e^{i} \otimes e_{j},
\end{equation}
where $i, j$ ranges from $1, \cdots, n, \overline{1}, \cdots, \bar{n}$. Then (\ref{difference of complex structure}) implies that we have $\left|J_{i}^{j}\right|=O\left(\left|w\right|\right)$ for all $i, j .$  Then the lemma follows from the explicit expression of the K\"ahler metric on $\mathcal D$, see (\ref{estimate for 1-forms}). 
\end{proof}

From the assumption (\ref{assumption on asymptotics}) on the K\"ahler form and (\ref{complex structure is close}) on the complex structure asymptotics, we obtain that for the corresponding Riemannian metric
\begin{equation}\label{closeness of metric tensor}
	\left|\Phi^*g-g_F\right|_{g_F}=O(r_F^{-N_0}).
\end{equation}

It is also useful to write down the Riemannian metric $g_F$ explicitly in real coordinates. Note that the set $\left\{\xi\in N_D:\left|\xi\right|_{h_D}<1\right\}$ is diffeomorphic to $\mathbb{R_{+}}\times Y$, where $Y$ is a smooth $(2n-1)$-dimensional $S^1$ bundle over $D$. Let $F(t)=At^a\in \mathcal H$. Then we can write the Riemannian metric $g_{F}$ as follows
\begin{equation}\label{metric tensors on real cosrdinates}
	g_{F}=dr^2+C_1r^{2(1-\frac{1}{a})}g_D+C_2r^{2(1-\frac{2}{a})}\theta^2,
\end{equation}
where $g_D$ is the corresponding Riemannian metric for $\omega_D$ and $\theta$ is a connection 1-form on $Y$ such that $d\theta=\omega_D$.

From the asymptotic of the Riemannian metric tensor (\ref{closeness of metric tensor}), the explicit expression of the Riemannian metric $g_F$ in (\ref{metric tensors on real cosrdinates}) and conditions in \textit{Assumption 1}, one can directly show the following result.
\begin{lemma}\label{vanishing at infinity decay} Suppose $(X,\omega,g)$ satisfy \textit{Assumption 1}, then 
\begin{itemize}
	\item [(1)] the volume growth of $g$ is at most 2, i.e. there exists a constant $C>0$ such that Vol($B_R(p))\leq CR^2$ for all $R$ sufficiently large.
	\item [(2)] for large numbers $K,\alpha=2$ and $\beta=\frac{4}{a}-2$, $(X,\omega)$ is of $(K,\alpha,\beta)$-polynomial growth  as defined in \cite[Definitino 1.1]{tian-yau1},
	\item [(3)]	if $\theta$ is a smooth form on $\olsi{X}$ vanishing when restricted to $D$, then $$\left|\theta\right|_{g}=O(r^{-N_0}).$$
\end{itemize}
\end{lemma}

	That $(M,g)$ is of $(K,\alpha,\beta)$-polynomial growth is important for us since we need the weighted Sobolev inequality in  \cite[Proposition 2.1]{tian-yau1} to prove a weighted mean value inequality in the next subsection. %And we can also show that $(M,g)$ is of $\operatorname{SOB}(\beta)$ as defined by Hein in \cite{hein2011} and use his version of weighted Sobolev inequality there.

\subsection{A weighted mean value inequality}
In this subsection, using a weighted Sobolev inequality in \cite{tian-yau1}, we prove a weighted mean value inequality for nonnegative functions which are almost subharmonic. This is important when we run Simpson's argument to get a uniform $C^0$-estimate. As usual, $r$ denotes the distance function to a fixed base point induced by a Riemannian metric.
\begin{lemma}\label{weighted sobolev inequality}
Let $(X,g)$ be a Riemannian manifold which is of $(K,\alpha,\beta)$-polynomial growth as defined in \cite{tian-yau1}. Let $u$ be a nonnegative compactly supported Lipschitz function satisfying $\Delta u \leq f$ in the weak sense. Suppose  that $|f|=O(r^{-N})$, for some $N\geq 2+\alpha+\beta$, then there exist $C_i=C_i(n, N)$ such that 
	\begin{equation}\label{mean value inequality}
		\left\Vert u\right \Vert_{L^{\infty}}\leq C_1\int(1+r)^{-N}u+C_2
	\end{equation}
\end{lemma}

\begin{proof}
	The following argument is the standard Moser iteration with the help of the weighted Sobolev inequality in  \cite[Proposition 2.1]{tian-yau1}.
	
Let $\gamma=\frac{2n+1}{2n-1}$. Note that we have $\int u^p\Delta u\ \leq \int u^pf$ for any $p\geq 1$. Integration by parts and using that $|f|=O(r^{-N})$, we have
	\begin{equation*}
		\int |\nabla u^{\frac{p+1}{2}}|^2 \leq Cp\int u^p(1+r)^{-N}.
	\end{equation*}
	 Let $d\mu=(1+r)^{-N}dV_g$ and without loss of generality, we may assume $d\mu$ has total mass 1. Then the weighted Sobolev inequality shows that 
	 \begin{equation*}
	 	\left(\int \left|u^{\frac{p+1}{2}}-\int u^{\frac{p+1}{2}}d\mu\right|^{2\gamma}d\mu\right)^{\frac{1}{2\gamma}}\leq C\left(\int |\nabla u^{\frac{p+1}{2}}|^2\right)^{\frac{1}{2}}\leq C p^{\frac{1}{2}}\left(\int u^p d\mu\right)^{\frac{1}{2}}.
	 \end{equation*}
	 Applying the triangle inequality and H\"{o}lder inequality, we get 
	 \begin{equation*}
	 \begin{aligned}
	 \left(\int u^{(p+1)\gamma}d\mu\right)^{\frac{1}{2\gamma}}&\leq C_1\int u^{\frac{p+1}{2}}d\mu+C_2p^{\frac{1}{2}}\left(\int u^pd\mu\right)^{\frac{1}{2}}\\
	 &\leq C_1\left(\int u^{p+1}d\mu\right)^{\frac{1}{2}}+C_2p^{\frac{1}{2}}\left(\int u^{p+1}d\mu\right)^{\frac{p}{2(p+1)}}.
	 \end{aligned}
	 \end{equation*}
	 Let $p_i=\gamma^i$, $i=0,1,\cdots$. We have for any $i$
	 \begin{equation*}
	 	\left(\int u^{p_{i+1}}d\mu\right)^{\frac{1}{\gamma}}\leq C_1 \int u^{p_i}d\mu+C_2p_i\left(\int u^{p_i}d\mu\right)^{\frac{p_i}{p_{i+1}}}.
	 \end{equation*}
	 Either there exists a sequence of $p_{i_j}\rightarrow \infty$ such that $\int u^{p_{i_j}}d \mu \leq 1$, which implies that $\Vert u\Vert_{L^{\infty}}\leq 1$, or there exists a smallest $i_0$ such that and $\int u^{p_{i}}d\mu > 1$ for $i\geq i_0$. In the second case, we have 

\begin{equation*} 
\begin{aligned}
	&\Vert u\Vert_{L^{p_{i_0}}}\leq \max\left\{ \Vert u\Vert_{L^1(d\mu)}, C p_{i_0}^{\frac{1}{p_{i_0}}}\right\}\leq C_1\Vert u\Vert_{L^1(d\mu)}+C_2\\
	&\left(\int u^{p_{i+1}}d\mu\right)^{\frac{1}{\gamma}}\leq C p_i\int u^{p_i}d\mu, \ \text{for}\ i\geq i_0.
\end{aligned}
\end{equation*}
Iterating gives that 
\begin{equation*}
	\Vert u\Vert_{L^{\infty}}=\lim_{i\rightarrow \infty}\Vert u\Vert_{L^{p_i}(d\mu)}\leq C \Vert u\Vert_{L^{p_{i_0}}}\leq C_1\Vert u\Vert_{L^1(d\mu)}+C_2
\end{equation*}	 
\end{proof}

\subsection{The assumption on the degree $a$}
The only reason why we need to assume  $a\leq \frac{n}{n-1}$ is that the volume growth of the corresponding Riemannian metric is at most 2. In fact we have the following easy but useful degree vanishing property for Rimannian manifolds with at most quadratic volume growth.
\begin{lemma}\label{exact integral imply degree 0}
	Let $(M,g)$ be a complete Riemannian manifold with volume growth order at most 2. Let $u\in C^{\infty}(M)$ satisfying $|\nabla u|\in L^2$ and $\Delta u\in L^1$, then 
	\begin{equation*}
		\int_{M} \Delta u \ dV_{g}=0.
	\end{equation*}
\end{lemma}

\begin{proof}
By the Cauchy-Schwarz inequality and the assumption on the volume growth, we have 
\begin{equation*}
	\frac{1}{R}\int_{B_{2R}\backslash B_R}|\nabla u|dV_g\leq C \left(\int_{B_{2R}\backslash B_R}|\nabla u|^2dV_g\right)^{\frac{1}{2}}\rightarrow 0 \text{ as $R\rightarrow \infty$} .
\end{equation*}

Therefore there is a sequence $R_i\rightarrow \infty$ such that   $\int_{\partial B_{R_i}}|\nabla u|\ dS\rightarrow 0$.
	Since $\Delta u$ is integrable, $\int_{M} \Delta u \ dV_{g}=\lim_{i\rightarrow\infty}\int_{B_{R_i}} \Delta u\ dV_g $ for any sequence $R_i$ going to infinity. Using Stokes' theorem, we have
	\begin{equation*}
		\left|\int_{B_{R_i}}\Delta u\ dV_g\right|\leq \int_{\partial B_{R_i}}|\nabla u|\ dS\rightarrow 0 \text{ as $R_i\rightarrow\infty$}.
	\end{equation*}
\end{proof}

\subsection{Assumption on $\Phi$ and $\omega$}
 %Firstly there do exist diffeomorphisms satisfying the two conditions assumed in \textit{Assumption 1}, i.e.
%\begin{equation*}
	%\text{$\Phi|_D=\operatorname{id}$ and $\pp\Phi=0$ on $D$},\end{equation*}
 %for example the exponential map given by fixing a back ground metric $\olsi{g}$ on $\olsi{X}$ and identifying $N_D$ with the $g$-orthogonal complement of $TD$ in $T\olsi{X}$. 

By Proposition \ref{basic geometric property} and the assumption on the decomposition of $\omega=\omega_0+\ii \partial\pp \varphi$, we know that (\ref{assumption on asymptotics}) is equivalent to say that 
\begin{equation}\label{approximate for exact part}
		\left|\Phi^*(\ii\partial\pp \varphi)-\omega_{F}\right|_{g_{F}}=O(r_{F}^{-N_0}).
\end{equation}
Writing $\Phi^*(\ii\partial\pp \varphi)-\omega_{F}$ as $d(\Phi^*d^c\varphi-d^cF)$ and integrating this exact 2-form,  we can show the following result, whose proof is similar to that given in \cite[Lemma 3.7]{hsvz1}.
\begin{lemma}\label{a good 1-form potential}
	There exists a real 1-form $\eta$ outside a compact set of $\mathcal C$ with $$|\eta|_{g_F}=O(r_{F}^{-N_0+1+\frac{1}{a}})$$ such that  
	\begin{equation*}
		\Phi^*(\ii\partial\pp \varphi)-\omega_{F}=d\eta
	\end{equation*}
\end{lemma}

\begin{proof}Choose a cut-off function $\chi$ which equals 1
	on $\{0<|\xi|_{h_D}< \delta\}$and 0 on $\{|\xi|_{h_D}> 2\delta\}$ for some $\delta>0$ . Let 
	\begin{equation*}
		\theta=d(\chi(\Phi^*(d^c \varphi)-d^cF))
	\end{equation*}
	And it suffices to write $\theta=d\eta$ with $|\eta|_{g_F}=O(r_{F}^{-N_0+1+\frac{2}{a}})$.

We identify $\mathcal C$ with $\mathbb{R_{+}}\times Y$ in such a way that the Riemannian metric $g_F$ can be written as $dr^2+g_r$, where $r$ is the coordinate function on $\mathbb R_{+}$ and $g_r$ is a metric on $\{r\} \times Y^{2n-1}$ that depends on $r$.
	 Then $\theta$ is supported on the region $\{r>r_0\}$ for some $r_0>0$. Then there exist a 1-form $\alpha$ and a 2-form $\beta$ supported on the region $\{r>r_0\}$ such that $\left.\partial_{r}\right\lrcorner \alpha=0$ and 
	$\left.\partial_{r}\right\lrcorner \beta=0$
	\begin{equation*}
		\theta=dr\wedge\alpha+\beta.
	\end{equation*}
Then we define 
\begin{equation*}
	\eta=\int_{r_0}^r\alpha\ dr.
\end{equation*}
$\theta$ is closed, therefore $d\alpha+\partial_r\beta=0$ and then one can directly check that $\theta=d\eta$. Since $dr\wedge \alpha$ is perpendicular to $\beta$ and we assumed $|\theta|_{g_F}=O(r_F^{-N_0})$, we obtain that $|\alpha|_{g_F}=O(r_F^{-N_0})$. Fix a smooth background Riemannian metric $\bar g$ on $Y$. Then from (\ref{metric tensors on real cosrdinates}) and (\ref{closeness of metric tensor}), we obtain the following estimate
\begin{equation*}
	C^{-1}r^{2(1-\frac{2}{a})}\bar g\leq g_r \leq C r^{2(1-\frac{1}{a})}\bar g.
\end{equation*} 
Then the estimate for $|\eta|_{g_F}$ follows from a direct computation.
\end{proof}

\begin{remark}\label{a new 1-form potential}
	A similar argument can be applied to $dd^c\varphi$ directly on $X$ (using \textit{Assumption 1}) and we obtain that there exists a cut-off function $\chi$ supported on a compact set and a smooth real 1-form $\psi$ supported outside a compact set satisfying $|\psi|=O(r^{1+\frac{1}{a}})$ such that 
	\begin{equation*}
		dd^c\varphi=dd^c(\chi\varphi)+d\psi
	\end{equation*} This is quite useful when we want to integrate by parts on $X$.
\end{remark}

 We assumed that $\omega_0$ is a closed (1,1)-form on $\olsi{X}$ and vanishes when restricted to $D$. In particular, $\int_{\olsi{X}}c_1(D)\wedge\omega_0^{n-1}=0$. Then by  the Lelong-Poincar\'e formula, we obtain the following. 

\begin{lemma}\label{poincare lelong}
	Let $S\in H^0(\olsi{X},L_D)$ be a defining section of $D$ and $h$ be any smooth hermitian metric on $L_D$. Let $f=\log |S|^2_h$, then we have
	\begin{equation*}
		\int_Xdd^cf\wedge \omega_0^{n-1}=0
	\end{equation*}
\end{lemma}

\section{Hermitian holomorphic vector bundles} \label{basics on vector bundle}
	Firstly let us recall the definition of projectively Hermitian-Yang-Mills metrics. Given a hermitian metric $H$ on a holomorphic vector bundle $E$, there is a unique connection compatible with these two structures and it is called the \textit{Chern connection} of $(E,H)$. Let $F_H$ denote the curvature of the \textit{Chern connection} and we call it the \textit{Chern curvature} of $(E,H)$. Let $E$ be a holomorphic vector bundle on a K\"ahler manifold $(X,\omega)$. A hermitian metric $H$ is called an $\omega$-projectively Hermitian-Yang-Mills metric ($\omega$-PHYM) if
\begin{equation}\label{definition of hym}
	 \Lambda_{\omega}F_H=\frac{\tr (\Lambda_{\omega}F_H)}{\rank (E)}\mathrm{id}_E,
\end{equation} Accordingly the Chern connection is called an $\omega$-PHYM connection if (\ref{definition of hym}) is satisfied.  A hermitian metric $H$ is called an $\omega$-Hermitian-Yang-Mills ($\omega$-HYM) metric if $$\Lambda_{\omega}F_H=\lambda \mathrm{id}_E$$ for some constant $\lambda$. We also use the notation $F^{\perp}_H$  to denote the trace-free part of the curvature form, i.e. $F_H^{\perp}=F_H-\frac{\tr (F_H)}{\rank (E)}\mathrm{id}_E$. Then (\ref{definition of hym}) is equivalent to say $\Lambda_{\omega}F^{\perp}_H=0$.

\begin{remark}
	Note that the PHYM property is conformally invariant, i.e. if a hermitian metric $H_0$ satisfies (\ref{definition of hym}), then $H=H_0e^f$ also satisfies (\ref{definition of hym}) for every smooth function $f$. Moreover to get a HYM metric from a PHYM metric, it suffices to solve a Poisson equation, which is always solvable for a constant $\lambda$ such that $\int_X(\tr(\Lambda_{\omega}F_H)-\lambda)\omega^n=0$ when $M$ is compact. 
\end{remark}

\subsection{Basic differential inequalities} Let $E$ be a holomorphic vector bundle and $H,K$ be two hermitian metrics on $E$, then we have an endomorphism $h$ defined by
\begin{equation*}
	\left<s,t\right>_H=\left<h(s),t\right>_K.
\end{equation*}
We will write this as $H=Kh$ and $h=K^{-1}H$ interchangeably. Note that $h$ is positive and self-adjoint with respect to both $H$ and $K$. Let $\partial_H$ and $\partial_K$ denote the $(1,0)$ part of the Chern connection determined by $H$ and $K$ respectively. By abuse of notation, we use the same notation to denote the induced connection on $\operatorname{End}( E)$. Simpson showed that 
\begin{lemma}[{\cite{simpson}}]\label{basic differential inequlaties}
Let $H=Kh$, then we have 
\begin{itemize}
	\item [(1)] $\partial_H=\partial_K+h^{-1}\partial_K(h)$;
	\item [(2)] $\Delta_{K}h=h\ii(\Lambda F_H-\Lambda F_K)+\ii \Lambda \pp (h) h^{-1}\partial_K(h)$ where $\Delta_K=\ii\Lambda\pp\partial_K$;
	\item [(3)]$\Delta \log \operatorname{tr}(h) \leq 2\left(\left|\Lambda F_{H}\right|_{H}+\left|\Lambda F_{K}\right|_{K}\right)$.
\end{itemize}
Moreover in (2) and (3), if $\det(h)=1$ then the curvatures can be replaced by the trace-free curvatures $F^{\perp}$.
\end{lemma}

\subsection{Slope stability}\label{definition of stability with respect to a pair}
 If $ \left|\tr(\Lambda_{\omega}F_H)\right| \in L^1$, the $\omega$-degree of $(E,H)$ and $\omega$-slope of $(E,H)$ are defined to be
	 \begin{equation}\label{definition of slope}
	 \begin{aligned}
	 	\deg_{\omega}(E,H)&=\frac{\ii}{2n\pi}\int_M \tr(\Lambda_{\omega}F_H)\omega^n=\frac{\ii}{2\pi}\int_M\tr(F_H)\wedge\omega^{n-1}\\
	 	\mu_{\omega}(E,H)&=\frac{\deg_{\omega}(E,H)}{\rank (E)}
	 \end{aligned}	 	
	 \end{equation} 
	 
Now let us assume $M$ is compact. Integration by parts shows that the degree defined above is independent of the metric $H$ and only depends on  the cohomology class of $[\omega]\in H^2(X,\mathbb R)$, i.e. by the Chern-Weil theory, 
\begin{equation*}
	\deg_{\omega}(E)=\int_Mc_1(E)\wedge\omega^{n-1}.
\end{equation*} And for any coherent subsheaf $S$ of $E$, one can define its $\omega$-degree as follows (see \cite[Chapter 5]{kobayashi-book}). It is shown that $\det S:=(\wedge^rS)^{**}$ is a line bundle, where $r$ is the rank of $S$ and define 
\begin{equation}\label{definition of degree of sheaf}
	\deg_{\omega}(S)=\int_Mc_1(\det S)\wedge\omega^{n-1}.
\end{equation}
And as before we define $\mu_{\omega}(S)$, the $\omega$-slope of $S$, to be $\frac{\deg_{\omega}(S)}{\rank (S)}$. Note that for the definition of $\omega$-degree and $\omega$-slope, we do not need $\omega$ to be a K\"ahler form at all, and a real closed $(1,1)$-form is enough. That is for every real closed $(1,1)$-form $\alpha$, we can define 
\begin{equation*}
	\deg_{\alpha}(S)=\int_Mc_1(\det S)\wedge \alpha^{n-1}.
\end{equation*}
The slope $\mu_{\alpha}(S)$ is defined similarly as before and we will use the notation $\mu(S,\alpha)$ and $\mu_{\alpha}(S)$ interchangeably.

We have the following definition which generalizes the standard slope stability defined for K\"ahler classes in  \cite[Chapter 5]{kobayashi-book}. 
\begin{definition}\label{definition of pair stability}
	Let $M$ be a compact K\"ahler manifold, $\alpha,\beta\in H^{1,1}(M)$ be two cohomology classes, $E$ be a holomorphic vector bundle over $M$.
	\begin{itemize}
		
\item [(1)]We say $E$ is $\alpha$-stable if for every coherent reflexive subsheaf $S$ of $E$ with $0<\rank(S)<\rank(E)$, we have $\mu_{\alpha}(S)<\mu_{\alpha}(E)$; $E$ is $\alpha$-polystable if it is the direct sum of stable vector bundles with the same $\alpha$-slope; $E$ is $\alpha$-semistable if for every coherent reflexive subsheaf $S$ of $E$ with $0<\rank(S)<\rank(E)$, we have $\mu_{\alpha}(S)\leq\mu_{\alpha}(E)$;. 

		\item [(2)] We say $E$ is $(\alpha,\beta)$-stable if every coherent reflexive subsheaf $S$ of $E$ with $0<\rank (S)<\rank (E)$ satisfies either of the following conditions:
	\begin{itemize}
		\item [(a)] $\mu_{\alpha}(S)< \mu_{\alpha}(E)$, or 
		\item[(b)] $\mu_{\alpha}(S)= \mu_{\alpha}(E)$ and $\mu_{\beta}(S)<\mu_{\beta}(E)$.
	\end{itemize}
	
	\end{itemize} 
\end{definition}

From the definition, we know that 
if $\beta=0$, then $E$ is $(\alpha,\beta)$-stable if and only if it is $\alpha$-stable; if $E$ is $\alpha$-stable, then it is $(\alpha,\beta)$-stable for any class $\beta$. In applications, typically the first class $\alpha$ has some positivity. For example, in our Theorem \ref{main theorem}, $\alpha=c_1(D)$ is nef and big.

\begin{remark}
For every coherent subsheaf $S$ of a holomorphic vector bundle $E$, we have an exact sequence of sheaves:
\begin{equation*}
	0\rightarrow S\rightarrow S^{**}\rightarrow S^{**}/S\rightarrow0,
\end{equation*} where $S^{**}/S$ is a torsion sheaf and supports on an analytic set with codimension at least 2. Then by \cite[Section 5.6]{kobayashi-book}, we know $\det S=\det (S^{**})$. In particular, we know that $E$ is $\alpha$-stable (respectively $(\alpha,\beta)$-stable) if and only if the conditions in (1) (respectively (2)) hold for every coherent subsheaf of $E$.

\end{remark}

%The method used in this paper to get a PHYM metric on a noncompact K\"ahler manifold $X$ is the same as that in \cite{simpson} and \cite{mochizuki}. We take a sequence of compact exhaustion $X_i$ for $X$ and solve the Dirichlet problem on each $X_i$ and prove a uniform $C^0$ estimate for solutions. Then Bando-Siu interior estimate will give us a local uniform estimate for all higher order derivative of $s_i$ Let $(Z,\omega)$ denote a (connected) compact K\"ahler manifold with boundary and let $(E,H_0)$ be a hermitian holomorphic vector bundle over $Z$.

\subsection{Coherent  subsheaves and weakly holomorphic projection maps}
Let $(E,H)$ be a hermitian holomorphic vector bundle over a K\"ahler manifold $(M,\omega)$. Suppose $S$ is a coherent subsheaf of $E$, since $E$ is torsion-free, then $S$ is torsion free and hence locally free outside $\Sigma$ which is a closed analytic set of codimension at least 2. Moreover on $X\backslash \Sigma$ we have an induced orthogonal projection map $\pi=\pi^H_S$ satisfying
\begin{equation}\label{weakly holomophic map condition}
\pi=\pi^{\star}=\pi^{2}, \quad(\mathrm{id}-\pi) \circ \pp \pi=0.
\end{equation}
Outside the singular set $\Sigma$, the Chern curvature of $(S,H|_S)$ is related to the Chern curvature $(E,H)$ by 

\begin{equation}\label{curvature form for subhseaf}
	F_{S,H}=F_{E,H}|_S-\partial \pi\wedge\pp\pi.
\end{equation}

Let us mention a result in current theory:
\begin{theorem}[{\cite{harvey1974removable}}]\label{current lemma}
	Let $\Sigma$ be a closed analytic subset of codimension at least 2 in a K\"ahler manifold $(M,\omega)$. Assume $T$ is a closed positive current on $M\backslash \Sigma$ of bidegree $(1,1)$, i.e a $(1,1)$-form with distribution coefficients, then the mass of $T$ is locally finite in a neighborhood of $\Sigma$. More precisely, every  $p\in \Sigma$ has a neighborhood $U\subseteq M$ such that 
	\begin{equation*}
		\int_U T\wedge\omega^{n-1}<\infty.
	\end{equation*}
\end{theorem}

Applying the above theorem to $\tr(\ii\partial\pi\wedge\pp \pi)$, one gets
\begin{equation}\label{L^2}
	\pi\in W^{1,2}_{loc}(M,\operatorname{End}( E);\omega,H).
\end{equation} 

And in general we call $\pi\in W^{1,2}_{loc}(M,\operatorname{End}( E);\omega,H)$ is a weakly holomorphic projection map if it satisfies (\ref{weakly holomophic map condition}) almost everywhere. By the discussion above, we know that for a coherent subsheaf $S$ of $E$, $\pi_{S}^H$ is a weakly holomorphic projection map. A highly nontrivial result due to Uhlenbeck and Yau \cite{uhlenbeck-yau} is that the converse is also true (see also \cite{popovici}). 

\begin{theorem}[{\cite{uhlenbeck-yau}}]\label{uhlenbeckyau}
	Suppose there is a weakly holomorphic projection map $\pi$, then there exists a coherent subsheaf $S$ of $E$ such that
	$\pi=\pi^H_{S}$ almost everywhere. 
\end{theorem}

If $X$ is compact, $\deg_{\omega}(S)$ defined in (\ref{definition of degree of sheaf}) can be computed using the curvature form $F_{S,H}$. The following result is well-known, see \cite[Section 5.8]{kobayashi-book}. We include a simple proof using Theorem \ref{current lemma}.
\begin{proposition}
	Let $(E,H)$ be a hermitian holomorphic vector bundle over a compact K\"ahler manifold $(M,\omega)$ and $S$ be a coherent subsheaf of $E$. Then 
	\begin{equation}\label{degree of coherent sheaf}
		\deg_{\omega}(S)=\frac{\ii}{2\pi}\int_{M\backslash\Sigma} \tr(F_{S,H})\wedge\omega^{n-1},
	\end{equation}
	where $\deg_{\omega}(S)$ and $F_{S,H}$ are defined in (\ref{definition of degree of sheaf}) and (\ref{curvature form for subhseaf}) respectively.
\end{proposition}

\begin{proof}
	Let $r$ denote the rank of $S$.  
Since $S$ is a subsheaf of $E$, there is a natural sheaf homomorphism 
\begin{equation*}
	\Phi:(\wedge^rS)^{**}\longrightarrow (\wedge^r E)^{**}=\wedge^r E.
\end{equation*}
Note that $\Phi$ is only injective on $M\backslash \Sigma$ in general. Let $\wedge^rH$ denote the metric on $\wedge^r E$ induced from $H$, then $\Phi^*(\wedge^rH)$ defines a singular hermitian metric on $(\wedge^rS)^{**}$ which is smooth outside $\Sigma$ and whose curvature form is equal to $\tr(F_{S,H})$. Since $\Phi$ is a holomorphic bundle map, by choosing a local holomorphic basis of $(\wedge^rS)^{**}$ and $\wedge^r E$, it is easy to show that $\Phi^*(\wedge^rH)=fK$, where $K$ is a smooth hermitian on $(\wedge^rS)^{**}$, the function $f$ is positive smooth outside $\Sigma$ and converge to 0 polynomially along $\Sigma$.

Then by Theorem \ref{current lemma}, it suffices to prove the following: for every smooth positive function $f$ on $M\backslash \Sigma$ satisfying $\Delta \log f\in L^1$, and $|f|=O(\operatorname {dist}(\cdot, \Sigma)^{k})$ for some $k$ one has
\begin{equation}\label{another average of laplacian is 0}
	\int_M\Delta \log f\ \omega^n=0.
\end{equation}
Note that $|\log f|\in L^2$, then (\ref{another average of laplacian is 0}) follows from the Cauchy-Schwarz inequality and existence of good cut-off functions. More precisely, since $\Sigma$ has real codimension at least 4, it is well-known that there exists a sequence of cut-off functions $\chi_{\epsilon}$ such that $1-\chi_{\epsilon}$ is supported in the $\epsilon$-neighborhood of $\Sigma$ and we have uniform $L^2$ bound on $\Delta \chi_{\epsilon}$. 

We briefly describe a construction of these cut-off functions. Let $s$ be a regularized distance function to $\Sigma$ in the sense that $s:M\rightarrow \mathbb R_{\geq0}$ is smooth and satisfies that there exist positive constants $C_k$ such that
\begin{equation*}
	C_0^{-1}\operatorname{dist}(x,\Sigma)\leq s(x)\leq C_0\operatorname {dist}(x,\Sigma) \text{ and } |\nabla^{k}s|\leq C_k\operatorname{dist}(x,\Sigma)^{1-k} \text{ for all } k\geq0.
\end{equation*}The existence of such a regularized distance function can be derived from \cite[Theorem 2 on page 171]{Stein}. After a rescaling, we may assume $s<1$ on $M$.
 For every $\epsilon>0$, let $\rho_{\epsilon}$ be a smooth function which is equal to one on the interval $(-\infty, \varepsilon^{-1})$ and zero on $(2+\varepsilon^{-1}, \infty)$. Moreover we can have $|\rho^{\prime}_{\epsilon}|+|\rho^{\prime\prime}_{\epsilon}|\leq 10$. Then we define
\begin{equation*}
	\chi_{\epsilon}=\rho_{\epsilon}(\log(-\log s)),
\end{equation*}
and we can directly check they satisfy the desired properties. 
\end{proof}

Motivated by the above result, Simpson \cite{simpson} uses the right hand side of (\ref{degree of coherent sheaf}) to define an analytic $\omega$-degree of a coherent subsheaf on a noncompact K\"ahler manifold.  Typically one needs to assume  $|\Lambda_{\omega}F_H|\in L^1$ to ensure the first term of (\ref{curvature form for subhseaf}) is integrable. Then the degree of a coherent subsheaf is either a finite number or $-\infty$ depending on whether $|\pp\pi|\in L^2$. In general, this analytic degree depends on the choice of the background metric $H$. And a key observation in this paper is that when $E$ has a compactification and $H$ is conformal to a smooth extendable hermitian metric, this analytic degree does have an algebraic interpretation, see Lemma \ref{degree coincide} and Lemma \ref{key observation}.

\subsection{Dirichlet problem}

	We have the following important theorem of Donaldson:
\begin{theorem} [{\cite{donaldsonboudary}}]\label{donaldson dirichlet}
Given a hermitian holomorphic vector bundle $(E,H_0)$ over $(Z,\omega)$ which is a compact K\"ahler manifold  with boundary, there is a unique hermitian metric $H$ on $E$ such that
\begin{itemize}
	\item [(1)] $H|_{\partial Z}=H_0|_{\partial Z}$,
	\item [(2)]$\ii\Lambda_{\omega}F_H=0$.
\end{itemize}
\end{theorem}

As observed in \cite{mochizuki}, one can do conformal changes to $H$ to fix the induced metric on $\det E$ and still have it to be a projectively Hermitian-Yang-Mills metric. 
\begin{proposition}[{\cite{mochizuki}}]
	Given a hermitian holomorphic vector bundle $(E,H_0)$ over  $(Z,\omega)$ which is a compact K\"ahler manifold  with boundary, there is a unique hermitian metric $H$ on $E$ such that
\begin{itemize}
	\item [(1)] $H|_{\partial Z}=H_0|_{\partial Z}$ and $\det H=\det H_0$,
	\item [(2)]$\ii\Lambda_{\omega}F_H^{\perp}=0$,
\end{itemize}

\end{proposition}

\subsection{Donaldson functional for manifolds with boundary}\label{definition of new endmorphisms} Next we recall Simpson's construction \cite{simpson} for Donaldson functional. We follow the exposition in \cite[Subsection 2.5]{mochizuki} and focus on compact K\"ahler manifolds with boundary.

Let $(Z,\omega)$ be a compact K\"ahler manifold with boundary, $(E,H_0)$ be a hermitian holomorphic vector bundle. Let $b$ be a smooth section of $\operatorname{End}(E)$ which is self-adjoint with respect to $H_0$. Then for any smooth function $f:\mathbb R\rightarrow \mathbb R$ and $\Phi:\mathbb R\times\mathbb R\rightarrow \mathbb R$ respectively, we can define 
\begin{equation*}
	\text{$f(b)\in C^{\infty}(\operatorname{End}( E))$ and $\Phi(b)\in C^{\infty}(\operatorname{End}(\operatorname{End}(E)))$ }
\end{equation*}as follows: at each point $p\in Z$, choose an orthonormal basis $\{e_i\}$ of $E$ such that $b(e_i)=\lambda_ie_i$. Let   $\{e_i^{\vee}\}$ denote its dual basis, then set

\begin{equation*}
	\text{$f(b)(e_i)=f(\lambda_i)e_i$ and 
$\Phi(b)(e_i^{\vee}\otimes e_j)=\Phi(\lambda_i,\lambda_j)e_i^{\vee}\otimes e_j$
 }
\end{equation*}

Recall that for any two hermitian metrics $H_1$ and $H_2$, there is a smooth section $s$ of $\operatorname{End}( E)$, which is self-adjoint with respect to both $H_1$ and $H_2$ such that $H_2=H_1e^s$ and $\det H=\det H_0$ if and only if $\tr(s)=0$. Let $\mathcal P_{H_0}$ denote the space of hermitian metrics  $H$ of $E$ such that 
\begin{equation*}
	\text{$\det H=\det H_0$ and $H|_{\partial Z}=H_0|_{\partial Z}$.}
\end{equation*}
Let $\Psi(\lambda_1,\lambda_2)$ denote the smooth function 
\begin{equation}\label{definition of Phi}
	\Psi(\lambda_1,\lambda_2)=\left\{
	\begin{aligned}
		&\frac{e^{\lambda_{2}-\lambda_{1}}-\left(\lambda_{2}-\lambda_{1}\right)-1}{\left(\lambda_{2}-\lambda_{1}\right)^{2}}\ \  \text{if $\lambda_1\neq \lambda_2$} \\
		&\frac{1}{2}\qquad \qquad\qquad \qquad\qquad\text{if $\lambda_1= \lambda_2$}.
	\end{aligned}
	\right.
\end{equation}
Then put

\begin{equation*}
	\mathcal M_{\omega} (H_1,H_2)=\ii\int_Z\tr(s\Lambda_{\omega}F_{H_1})\omega^n+\int_Z \left<\Psi(s)(\pp s),\pp s\right>_{H_1}\omega^n.
\end{equation*}

Mochizuki proved the following important result
\begin{theorem}[{\cite{mochizuki}}]\label{phym metrics minimizing donaldson functional}
	If $H\in \mathcal P_{H_0}$ is an $\omega$-PHYM metic, then $\mathcal M_{\omega}(H_0,H)\leq 0$.
\end{theorem}

\subsection{Bando-Siu's interior estimate}
The following result shows that to get a local uniform bound for a sequence of PHYM metrics, it suffices to have a uniform $C^0$-bound.
\begin{theorem}[{\cite{bando-siu}\cite{jacob-walpuski}}]\label{bando-siu interior estimate}
	Let $ B_2(p)\subseteq (M,\omega)$ be a geodesic ball of radius 2 contained in a K\"ahler manifold such that $\ols{B_2(p)}$ is compact. Let $(E,H_0)$ be a hermitian holomorphic vector bundle over $B_2(p)$. Suppose $H=H_0e^s$ is an $\omega$-PHYM metric with $\tr(s)=0$, then for any $k\in \mathbb N$ and $p\in(1,\infty)$ there exists a smooth function $\mathcal F_{k,p}$ which vanishes at the origin and depends only on the geometry of $B_2(p)$ such that 
	\begin{equation*}
		\left\Vert\nabla_{H_0}^{k+2}s\right\Vert_{L^p(B_1(p))} \leq \mathcal F_{k,p}\left(\left\Vert s\right\Vert_{L^{\infty}}+\sum_{i=0}^{k}\left\Vert\nabla_{H_0}^{i}F_{H_0}\right \Vert_{L^{\infty}(B_2(p))} \right).
	\end{equation*}
\end{theorem}

%%%%%%%%%%%%%%%%%%%%%%%%%%%%%%%%%%%%%%%%%%%%%%%%%%%%%%%
%%%%%%%%%%%%%%%%%%%%%%%%%%%%%%%%%%%%%%%%%%%%%%%%%%%%%%%
%%%%%%%%%%%%%%%%%%%%%%%%%%%%%%%%%%%%%%%%%%%%%%%%%%%%%%%
\section{Existence of a good initial metric}\label{existence of good initial metric}
In this section, we continue to use notations in Section \ref{assumption on the asymptotic behaviour} and always assume the K\"ahler metric on $X$ satisfies the \textit{Assumption 1}. 
We begin by working on the model space $(\mathcal C,\omega_{\mathcal C})$, where $\omega_{\mathcal C}=dd^c F(t)$ for some potential $F(t)\in \mathcal H$.
Using the explicit expression of $\omega_{\mathcal C}$ in (\ref{explicit expression of kahler form}), it is easily to show that
\begin{lemma} Let $E$ be a holomorphic vector bundle on $D$ and $p^*(E)$ be its pull back to $\mathcal C$. Suppose $H_D$ is a metric on $E$ satisfying $\ii\Lambda_{\omega_D}F_{H_D}=\lambda \mathrm{id}_E$ for some constant $\lambda$, then $H=e^{-\frac{\lambda}{n-1} \log |\xi|^2_{h_D}}$$p^*(H_D)$ defines a metric on $p^*(E)$ satisfying $\ii\Lambda_{\omega_{\mathcal C}}F_H=0$.
\end{lemma}

\begin{remark}\label{n=2 initial metric}
	If $n=2$, the metric $H$ actually is a flat metric on $p^*(E)$.
\end{remark}

 Let $\olsi E$ be a holomorphic vector bundle on $\mathcal D$. We still use $p$ to denote the projection map from $\mathcal D$ to $D$. Then we can compare the holomorphic structure on $\olsi E$ and $p^*(\olsi E|_D)$ as follows. In a neighborhood of $D$, which we may assume to be $\mathcal C$, we fix a bundle map $\Phi:\olsi E\rightarrow p^*(\olsi E|_D)$ such that $\Phi|_D$ is the canonical identity map and $\Phi$ is an isomorphism as maps between complex vector bundles. Then $\Phi$ pulls back the holomorphic structure on $p^*(\olsi E|_D)$ to $\olsi E$. Now we have two holomorphic structures on $\olsi E$ and denote them by $\pp_1$ and $\pp_2$. Then the difference $$\beta=\pp_2-\pp_1$$ is a smooth section of $\mathcal A^{0,1}(\operatorname{End}( \olsi E))$ and is 0 when restricted to $D$. Locally near a point in $D$, choose holomorphic coordinates $\{z_1,\cdots,z_{n-1},w\}$ such that $D=\{w=0\}$. Using these coordinates, $\beta$ can be written as $f_i\dd \bar z^i+h\dd \bar w$ where $f_i$ and $h$ are smooth sections of $\operatorname{End}( \olsi E)$ and $f_i|_{w=0}=0$.

Now suppose we have a Hermitian metric $H$ on $p^*(\olsi E|_D)|_{\mathcal C}$, then via $\Phi$ we view it as a metric on $\olsi E|_{\mathcal C}$. Let $\partial_i$ denote the $(1,0)$ part of the Chern connection determined by $\pp_i$ and $H$. Then one can check that $$\mu=\partial_2-\partial_1=-\beta^{*_H},$$ where $\beta^{*_H}$ denote the smooth section of $\mathcal A^{1,0}(\operatorname{End}( \olsi E))$ obtained from $\beta$ by taking the metric adjoint for the $\operatorname{End}( \olsi E)$ part and taking conjugate for the 1-form part. Locally $\mu=f_i^{*_H}\dd z^i+h^{*_H}\dd w$. Since $H$ is only defined over $\mathcal C$, a priori $f^{*_H}$ and $h^{*_H}$ are only defined on $\mathcal C$. Note that the operator $*_H$ is conformally invariant with respect to the choice of metric $H$, so if $H= e^f\olsi H$ for some function $f\in C^{\infty}(\mathcal C)$ and metric $\olsi H$ on $\olsi E$, $f_i^{*_H}$ and $h^{*_H}$ are smooth sections over $\mathcal D$. Then we can compute the difference of curvatures for $(\pp_1,H)$ and $(\pp_2, H)$:
\begin{equation}\label{difference of curvature}
F_{\pp_2,H}=\pp_2\partial_2+\partial_2\pp_2=F_{\pp_1,H}+\partial_1\beta+\pp_1\mu+[\beta,\mu],
\end{equation}
where we abuse notation to use the same symbol $\partial$ and $\pp$ to denote the induced connection on $\operatorname{End}( \olsi E)$. Again note that the induced metric (and hence the Chern connection) on $\operatorname{End}( \olsi E)$ is conformally invariant for  metrics $H$ on $\olsi E$. Therefore $F_{\pp_2,H}-F_{\pp_1,H}$ is a smooth $\operatorname{End}( \olsi E)$ valued $(1,1)$-form over $\mathcal D$ and $i^*_D(F_{\pp_2,H}-F_{\pp_1,H})=0$. Then by Proposition \ref{basic geometric property}, we obtain there exists a $\delta>0$,
	\begin{equation*}
		\left|F_{\pp_2,H}-F_{\pp_1,H}\right|_{\omega_{\mathcal C},H}=O(e^{-\delta t}).
	\end{equation*}
Combining this and the previous lemma, we proved that 

\begin{proposition} \label{model case}
	Let $\olsi E$ be a holomorphic vector bundle on $\mathcal D$. Suppose there is a metric $H_D$  on $\olsi E|_D$ satisfying $\ii\Lambda_{\omega_D}F_{H_D}=\lambda \operatorname{id}$ for some constant $\lambda$, then there is a Hermitian metric $H$ on $\olsi E|_{\mathcal C}$ satisfying:
	\begin{itemize}
		\item [(1).] there is a hermitian metric $\olsi{H}$ on $\olsi{E}$ and a function $f\in C^{\infty}(\mathcal C)$ such that $H=e^f\olsi{H}$, and
	\item [(2).] $|\Lambda_{\omega_{\mathcal C}}F_{H}|=O( e^{-\delta t})$ for some $\delta>0$.
	\end{itemize}
		\end{proposition}

Motivated by this, we can give the proof of the Lemma \ref{good initial metric}.

\noindent\textit{Proof of Lemma \ref{good initial metric}.}
	By  the Donaldson-Uhlenbeck-Yau theorem there exists a hermitian metric $H_D$ on $\olsi{E}|_D$ such that 
\begin{equation}\label{hermitian einstein condition}
	\ii\Lambda_{\omega_D}F_{H_D}=\lambda \mathrm{id}.
\end{equation} Extend $H_D$ smoothly to get a hermitian metric $\olsi{H}_0$ on $\olsi{E}$. Using the diffeomorphism  $\Phi$ given in the \textit{Assumption 1}-(3)  we get a positive smooth function on $X$ by abuse of notations stilled denoted by $t$, which is equal to $(\Phi^{-1})^{*}t$ outside a compact set on $X$.

 Define a hermitian metric on $E$ using
 \begin{equation}\label{definition of H_0}
 	 H_0=e^{\frac{\lambda}{n-1}t}\olsi{H}_0.
 \end{equation}
 Then we claim that 
 \begin{equation}\label{decay of contraction of curvature}
 	|\Lambda_{\omega}F_{H_0}|=O(r^{-N_0}).
 \end{equation}From the construction, $F_{H_0}=F_{\olsi{H}_0}-\frac{\lambda dd^c t}{n-1}\operatorname{id}$, where $F_{\olsi{H}_0}$ is smooth bundle valued (1,1)-form on $\olsi{X}$. Recall that for a 2 form $\theta$,
 \begin{equation*}
 	\Lambda_{\omega}\theta=\frac{n\theta\wedge\omega^{n-1}}{\omega^n}.
 \end{equation*} 
Since we assume that $|\Phi^*(\omega)-\omega_{\mathcal C}|=O(r^{-N_0})$, then (\ref{decay of contraction of curvature}) will follow from the following estimate on $\mathcal C$: there exist a $\delta>0$ such that 
\begin{equation}\label{estimate on model case}
	\left|\Lambda_{\omega_{\mathcal C}}\left(\Phi^*(F_{\olsi{H_0}})-\frac{\lambda}{n-1}\Phi^*dd^c t\right)\right|=O(e^{-\delta t}).
\end{equation} 
By (\ref{complex structure is close}) and (\ref{difference after 2.6}), we can easily show that there exists a $\delta>0$ such that
\begin{equation}\label{estimate for exact part}
	\left|\Phi^*dd^ct-dd^ct\right|=|d((\Phi^*J_{\olsi{X}}-J_{\mathcal D})\circ dt)|=O(e^{-\delta t}).
\end{equation}
Using the same argument as we did before Proposition \ref{model case}, we can show that there exists a $\delta>0$ such that
\begin{equation}\label{estimate for smooth bundle part}
	\left|\Phi^*(F_{\olsi{H}_0})-p^*(F_{\olsi{H}_0}|_D)\right|=O(e^{-\delta t}).
\end{equation}
Then (\ref{estimate on model case}) follows from (\ref{hermitian einstein condition}), (\ref{estimate for exact part}) and (\ref{estimate for smooth bundle part}).
\qed

\begin{remark}\label{t can be viewed}
	 Recall that we use $S\in H^0(\olsi{X},L_D)$ to denote a defining section of $D$. Then from the definition of $t$, we know that there exists a smooth hermitian metric $h$ on $L_D$ such that $t=-\log|S|_{h}^2$.
\end{remark}

\begin{remark}\label{full  curvature decay}
	From the above discussion, we also obtain that $|F_{H_0}|=O(r^{1-a})$. In general, we can not expect a higher decay order for the full curvature tensor $F_{H_0}$ since it has non-vanishing component along the directions tangential to $D$, but if $n=2$ we actually proved that 
	\begin{equation}\label{dim 2 case}
		|F_{H_0}|_{\omega}=O(r^{-N_0}).
	\end{equation}
\end{remark}

From the proof given above, the assumption that $E|_D$ is $c_1(N_D)$-polystable is used crucially to have a good initial metric $H_0$ satisfying (1) and (2) in Lemma \ref{good initial metric}, which both are important for the proof. We show the assumption that $E|_D$ is $c_1(N_D)$-polystable is also necessary subject to the conditions in Lemma \ref{good initial metric}. More precisely, we have that
\begin{proposition}\label{necessarity}
		Suppose there is a Hermitian metric $H_0$ on $E$ satisfying:
		\begin{itemize}
			\item [(1).] $|\Lambda_{\omega}F_{H_0}|=O(r^{-N_0})$, where as before $N_0$ is the number in (\ref{assumption on asymptotics}), and
	\item[(2).] there is a hermitian metric $\olsi{H}_0$ on $\olsi{E}$ and a function $f\in C^{\infty}(X)$ such that $H_0=e^f\olsi{H}_0$.
		\end{itemize}
Then $\olsi H_0|_D$ defines a PHYM metric with respect to $\omega_D\in c_1(N_D)$, i.e, $\olsi E|_D$ is $c_1(N_D)$-polystable.
\end{proposition}

\begin{proof}
	By these two assumptions, we have 
		$\left|\ii\Lambda_{\omega}F_{\olsi H_0}+\Delta_{\omega}f \operatorname{id} \right|=O(r^{-N_0})$. In particular the trace-free part of $\Lambda_{\omega}F_{\olsi H_0}$  decays like $r^{-N_0}$. Since $$F_{\olsi H_0}^{\perp}=F_{\olsi H_0}-\frac{\tr (F_{\olsi H_0})}{\rank E} \operatorname{id}$$ is a smooth bundle valued (1,1)-form on $\olsi X$, its pull-back under $\Phi$ to $\mathcal D$ is a smooth bundle valued 2-form satisfying that its restriction to $D$ is of type (1,1) and $$|\Lambda_{\omega_{F}}\Phi^*(F_{\olsi{H}_0}^{\perp})|=O(r^{-N_0}).$$ From the explicit expression of the K\"ahler form $\omega_{F}$ in (\ref{explicit expression of kahler form}) and the assumption on the potential $F$ in \textit{Assumption 1}, we know that $$\Lambda_{\omega_D}(F_{\olsi{H}_0}^{\perp}|_D)=0.$$
	\end{proof}

Next we can show that the set $\mathcal{P}_{H}$ defined in (\ref{define of P_H}) is unique if we fix the induced metric on the determinant line bundle. More precisely, we have 
\begin{proposition}\label{uniqueness of P_H}
	Suppose we have two metrics $H_0$ and $H_1$ satisfying the condition (1) and (2) in Lemma \ref{good initial metric} and $\det H_0=\det  H_1$, then 
	\begin{equation*}
		\mathcal P_{H_0}=\mathcal P_{H_1}.
	\end{equation*}
	In particular, there exists a constant $C>0$ such that 
		$C^{-1}H_0\leq H_1\leq CH_0$.
\end{proposition}
	
\begin{proof}
	By condition (1) in Lemma \ref{good initial metric}, we know that there are smooth hermitian metrics $\olsi{H}_0$ and $\olsi{H}_1$ and smooth functions $f_0$ and $f_1$ on $X$ such that for $i=0,1$	\begin{equation*}
		 H_i=\olsi{H}_ie^{f_i}.
	\end{equation*}
And by doing a conformal change, we may assume $\det \olsi{H}_0=\det \olsi{H}_1$. Let $h=\olsi{H}_0^{-1}\olsi{H}_1$. Since $\det H_0=\det H_1$, 
we have \begin{equation*}
	H_1=H_0h.
\end{equation*}
	 From the proof of Proposition \ref{necessarity}, we know that $\olsi{H}_i|_D$ are PHYM. By the uniqueness of PHYM metrics on compact K\"ahler manifolds, we know that
	 \begin{equation*}
	 		\nabla (h|_D)=0,
	 \end{equation*}
	 where $\nabla$ denotes the induced connection on $\operatorname{End}(\olsi{E})$ from the Chern connection on $(\olsi{E},\olsi{H}_0)$.
From this and noting that $H_0$ is conformal to an extendable metric, we can check directly that 
\begin{equation*}
	|\nabla h|\in L^2(X;\omega,H_0)
\end{equation*}
	Then from the definition of $\mathcal P_{H_0}$, we obtain that $\mathcal P_{H_0}=\mathcal P_{H_1}.$
\end{proof}

%%%%%%%%%%%%%%%%%%%%%%%%%%%%%%%%%%%%%%%%%%%%%%%%%%%%%%%
%%%%%%%%%%%%%%%%%%%%%%%%%%%%%%%%%%%%%%%%%%%%%%%%%%%%%%%
%%%%%%%%%%%%%%%%%%%%%%%%%%%%%%%%%%%%%%%%%%%%%%%%%%%%%%%

%%%%%%%%%%%%%%%%%%%%%%%%%%%%%%%%%%%%%%%%%%%%%%%%%%%%%%%
%%%%%%%%%%%%%%%%%%%%%%%%%%%%%%%%%%%%%%%%%%%%%%%%%%%%%%%
%%%%%%%%%%%%%%%%%%%%%%%%%%%%%%%%%%%%%%%%%%%%%%%%%%%%%%

%%%%%%%%%%%%%%%%%%%%%%%%%%%%%%%%%%%%%%%%%%%%%%%%%%%%%%%
%%%%%%%%%%%%%%%%%%%%%%%%%%%%%%%%%%%%%%%%%%%%%%%%%%%%%%%
%%%%%%%%%%%%%%%%%%%%%%%%%%%%%%%%%%%%%%%%%%%%%%%%%%%%%%%

\section{Proof of the main theorem}\label{proof of the main theorem}

We first prove a lemma on the degree vanishing property. 
\begin{lemma} \label{L1 imply degree 0}
	Let $(X^n, \omega)$ be a complete K\"{a}hler manifold and $\beta$ be a d-closed $(k,k)$ form with $\displaystyle \int_{X} |\beta|\omega^n$ finite for some $1\leq k\leq n-1$.  Suppose $\omega=d\eta$ for some smooth 1-form $\eta$ with $|\eta|=O(r)$, then $\displaystyle\int \beta\wedge \omega^{n-k}=0$.
\end{lemma}

\begin{proof}
	Fix a base point $p\in X$ and let $\rho_R$ be a smooth cut-off function which is 1 on $B_R(p)$, 0 outside $B_{2R}(p)$ and $|\nabla \rho_R|\leq \frac{C}{R}$ where $C$ is a constant independent of $R$.  Integrating by parts, we have the following 
		\begin{equation*}
		\begin{aligned}
			\left|\int_X \beta\wedge\omega^{n-k}\right|&=\left|\lim_{R\rightarrow\infty}\int_{B_{2R}(p)} \rho_R\beta\wedge\omega^{n-k}\right|\\
			&\leq C\lim_{R\rightarrow\infty}\int_{B_{2R}(p)\backslash B_R(p)}\frac{1}{R} \left|\beta\wedge \omega^{n-k-1}\wedge \eta\right|\omega^n,
		\end{aligned}
			\end{equation*}
	which is bounded by $C\int_{B_{2R}(p)\backslash B_R(p)}\left|\beta\right|\omega^n$. And this term tends to 0 as $R\rightarrow\infty$ since $|\beta| \in L^1$.
\end{proof}

%%%%%%%%%%%%%%%%%%%%%%%%%%%%%%%%%%%%%%%%%%%%%%%%%%%%%%%
%%%%%%%%%%%%%%%%%%%%%%%%%%%%%%%%%%%%%%%%%%%%%%%%%%%%%%%
%%%%%%%%%%%%%%%%%%%%%%%%%%%%%%%%%%%%%%%%%%%%%%%%%%%%%%
From now on, we assume $\omega=\omega_0+dd^c\varphi$ is a K\"ahler form satisfying the \textit{Assumption 1} in Section \ref{assumption on the asymptotic behaviour}. Note that we only proved that $dd^c\varphi=d\psi$ for a smooth form $\psi$ with $|\psi|=O(r^{1+\frac{1}{a}})$ (see Remark \ref{a new 1-form potential}). Therefore we can not apply Lemma \ref{L1 imply degree 0} directly. Typically we have a definite decay order for $|\beta|$, so we can still use integration by parts to show some degree vanishing properties. More precisely, we have

\begin{lemma}\label{a more precise degree vanishing}
	Let $\beta$ be a d-closed $(k,k)$ form for some $1\leq k\leq n-1$, satisfying $|\beta|=O(r^{-N_0})$. Then
	\begin{equation*}
		\int_X\beta\wedge(dd^c\varphi)^{n-k}=0
	\end{equation*}
\end{lemma}

\begin{proof}
	By a similar integration by part argument as in the proof of Lemma \ref{L1 imply degree 0}, it suffices to show that  
	\begin{equation*}
		\lim_{R\rightarrow\infty}\frac{1}{R} \int_{B_{2R}(p)\backslash B_R(p)}\left|\beta\wedge (dd^c\varphi)^{n-k-1}\wedge \psi\right|(dd^c\varphi)^n=0.
	\end{equation*}
	This follows from the facts that $|\beta|=O(r^{-N_0})$,  $|\psi|=O(r^{1+\frac{1}{a}})$ and the volume growth order of $\omega$ is at most 2.
\end{proof}

The following two lemmas are crucial for us since they relate information on $\olsi X$ and that on $X$.
\begin{lemma}\label{degree coincide}Let $H_0$ be the metric constructed in Lemma \ref{good initial metric}. One has the following equality:
	\begin{equation}\label{degree equality}
		\int_{X}\frac{\ii}{2\pi}\tr(F_{H_0})\wedge \omega^{n-1}=\int_{\olsi X}c_1(\olsi E)\wedge {[\omega_0]}^{n-1}.
	\end{equation}
\end{lemma}

\begin{proof} Firstly, recall that$$n\tr(F_{H_0})\wedge \omega^{n-1}=\Lambda_{\omega}\tr(F_{H_0})\omega^n.$$
By the construction in Lemma \ref{good initial metric}, we know that $|\Lambda_{\omega}\tr(F_{H_0})|=O(r^{-N_0})$. Since the volume growth order of $\omega$ is at most 2, we know that $\Lambda_{\omega}\tr(F_{H_0})$ is absolutely integrable. Therefore the left hand side of (\ref{degree coincide}) is well-defined.

 By the Chern-Weil theory, for any smooth hermitian metric $\olsi{H}_0$ on $\ols{E}$ we have
	\begin{equation*}
		\int_{\olsi X}c_1(\olsi E)\wedge {[\omega_0]}^{n-1}=\int_{\olsi X}\frac{\ii}{2\pi}\tr (F_{\olsi H_0})\wedge \omega_0^{n-1}.
	\end{equation*}
	By the construction \eqref{definition of H_0}, $H_0=e^{Ct}$$
\olsi{H}_0$ for some constant $C$ and $t$ defined in Section \ref{existence of good initial metric}. Moreover by Remark \ref{t can be viewed}, $t=-\log |S|^2_h$ for some smooth hermitian metric on $L_D$. By Lemma \ref{poincare lelong}, we obtain that $\int_X dd^ct\wedge \omega_0^{n-1}=0$. So we have
	\begin{equation}\label{a middle equality}
		\int_{\olsi X}c_1(\olsi E)\wedge {[\omega_0]}^{n-1}=\int_{X}\frac{\ii}{2\pi}\tr (F_{H_0})\wedge \omega_0^{n-1}.
	\end{equation}
	Using (\ref{a middle equality}) and $\omega=\omega_0+dd^c \varphi$, to prove (\ref{degree equality}), it suffices to show that for any $k=1,\cdots,n-1$,
	\begin{equation}\label{to be proved}
		\int_X\tr(F_{H_0})\wedge \omega_0^{n-1-k}\wedge(dd^c\varphi)^{k}=0.
	\end{equation}
	\noindent\textit{Case 1.} $1\leq k\leq n-2$. Since $\omega_0$ vanishes when restricted to $D$, by Lemma \ref{vanishing at infinity decay}, we know that $|\omega_0|=O(r^{-N_0})$. Combining this with Remark \ref{full  curvature decay}, we know that $\tr(F_{H_0})\wedge \omega_0^{n-1-k}$ is a closed $(n-k,n-k)$-form with decay order at least $r^{-N_0}$. Therefore Lemma \ref{a more precise degree vanishing} implies that its integral is 0.\\
	\noindent\textit{Case 2. }$k=n-1$. If $n=2$, then by (\ref{dim 2 case}) we can still apply Lemma \ref{a more precise degree vanishing}. If $n\geq 3$, note that though $|\Lambda_{\omega}\tr(F_{H_0})|=O(r^{-N_0})$, $|\tr(F_{H_0})|$ is not in $L^1$ in general. So we can not apply Lemma \ref{a more precise degree vanishing} directly. Instead we shall use the asymptotic behaviour of $\tr(F_{H_0})$ obtained from the construction.  Integrating by parts and pulling back via $\Phi$, we know that 
	\begin{equation*}
	\begin{aligned}
		\int_X\tr(F_{H_0})\wedge(dd^c\varphi)^{n-1}&=\lim_{\epsilon_i\rightarrow0}\int_{\Phi(|\xi|_{h_D}=\epsilon_i)}\tr(F_{H_0})\wedge (dd^c \varphi)^{n-2}\wedge d^c\varphi\\
		&=\lim_{\epsilon_i\rightarrow0}\int_{|\xi|_{h_D}=\epsilon_i}\Phi^*\left(\tr(F_{H_0})\wedge(dd^c\varphi)^{n-2}\wedge d^c\varphi\right)
	\end{aligned}
	\end{equation*}
	Then by (\ref{assumption on asymptotics}), Lemma \ref{a good 1-form potential} and the assumption $N_0>8$, we obtain that the right hand side of the above equality equals
	\begin{equation*}
		\lim_{\epsilon_i\rightarrow0}\int_{|\xi|_{h_D}=\epsilon_i}\Phi^*(\tr(F_{H_0}))\wedge(dd^cF(t))^{n-2}\wedge d^cF(t)
	\end{equation*}
	By (\ref{estimate for exact part}) and (\ref{estimate for smooth bundle part}), we know that it equals 
	\begin{equation}\label{almost done}
		\lim_{\epsilon_i\rightarrow0}\int_{|\xi|_{h_D}=\epsilon_i}\left(p^*\tr(F_{\olsi{H}_0})-\frac{\lambda \rank(E)}{n-1}
	dd^ct\right) \wedge(dd^cF(t))^{n-2}\wedge d^cF(t)
	\end{equation}
	Note that when restricted to the level set of $t$, $$(dd^cF(t))^{n-2}=F'(t)^{n-2}\omega_D^{n-2}.$$ Therefore $\left(p^*\tr(F_{\olsi{H}_0})-\frac{\lambda \rank(E)}{n-1}
	dd^ct\right) \wedge(dd^cF(t))^{n-2}=0$ by (\ref{hermitian einstein condition}). 
	\end{proof}

\begin{lemma}\label{key observation} Suppose $\olsi{E}|_D$ is $c_1(N_D)$-polystable and let $H_0$ be the metric constructed in Lemma \ref{good initial metric}. 	\begin{itemize}
		\item [(1).] Let $\olsi S$ be a coherent reflexive subsheaf of $\olsi E$.  If $\olsi S|_D$ is locally free and  a splitting factor of $\olsi E|_D$, then $\pp \pi_S^{H_0}\in L^2(X;\omega, H_0)$. 
		\item [(2).] Let $\pi\in W^{1,2}_{loc}(X,\olsi E^*\otimes\olsi E;\omega,H_0)$ be a weakly holomorphic projection map. If $\pp \pi \in L^2(X;\omega,H_0)$, then there exists a coherent reflexive subsheaf $\olsi S$ of $\olsi E$ such that $\pi=\pi^{H_0}_S$ a.e. and $\olsi S|_D$ is a splitting factor of $\olsi E|_D$.	 
	\end{itemize}
\end{lemma}

\begin{proof} A crucial point here is that $H_0$ is conformal to a smooth extendable metric $\olsi H_0$. In particular, for a coherent subsheaf $S$ of $E$, the projections induced by $H_0$ and $\olsi H_0$ are the same. Note that by \cite[Lemma 3.23 and Remark 3.25]{chen-sun}, for every coherent reflexive subsheaf $\olsi{S}$ of $\olsi{E}$, $\olsi{S}|_D$ is torsion free and can be naturally viewed as a subsheaf of $\olsi{E}|_D$.

	(1) Let $\pi=\pi_{\olsi S}^{\olsi H_0}$. Then  $\pi$ is smooth in a neighborhood of $D$ and $\pp \pi|_D=0$ by assumption. Note that $\pi_S^{H_0}=\pi|_X$, so it suffices to show $\pp \pi \in L^2(X,\omega, \olsi{H}_0)$. Fix small balls $U_i$ of $\olsi X$ covering $D$ such that there are holomorphic coordinates $\{z_1,\cdots,z_{n-1},w\}$ on each $U_i$ with $D\cap U_i=\{w=0\}$ and $\olsi E$ is trivial on each ball $U_i$. Under these coordinates and trivializations we can write $$\pp \pi=\pp_z\pi \dd \bar z+\pp_w\pi \dd \bar w,$$ where we view $\pp_z \pi$ and $\pp_w \pi$ as matrices of smooth functions and $\pp_z \pi|_{w=0}=0$. So we have $|\pp_z\pi|\leq C|w|$ and $|\pp_w\pi|\leq C$. Then the result follows from the explicit estimate given in (\ref{estimate for 1-forms}).
	
	(2) Given a projection map $\pi\in W^{1,2}_{loc}(X,\olsi E^*\otimes\olsi E;\omega,H_0)$ with $ \pp \pi \in L^2(X;\omega,H_0)$, we first prove the following:
	\begin{claim}
		 $\pi\in W^{1,2}(\olsi X;\omega_{\olsi X}, \olsi H_0)$ for a  fixed (hence any) smooth K\"ahler metric $\omega_{\olsi X}$ on $\olsi X$.
	\end{claim}
	
	Since $|\pi|_{\olsi H_0}\leq 1$ and by \cite[Lemma 7.3]{demailly-big-book}, it suffices to show $\pp\pi\in L^2(\olsi X; \omega_{\olsi X}, \olsi H_0)$. By (\ref{complex structure is close}) and (\ref{closeness of metric tensor}), we may assume in local coordinates around $D$ the K\"ahler metric $\omega$ is exactly given by the model space.  We choose local holomorphic coordinates $\underline{z}=\{z_i\}_{i=1}^{n-1}$ on the smooth divisor $D$ and fix a local holomorphic trivialization $e_0$ of $N_D$ with $|e_0|_{h_D}=e^{-\psi}$, where $\psi$ is a smooth function on $D$ satisfying $\ii \partial\pp \psi=\omega_D$. Then we get local holomorphic coordinatea $\{z_1,\cdots,z_{n-1},w\}$ on $\mathcal C$ by writing a point $\xi=we_0(\underline z)$. Choose a basis 
	of $(0,1)$-forms $\dd \bar z_1 \cdots\dd \bar z_{n-1}, \frac{\dd \bar w}{\bar w}-\pp_{z_i}  \psi\dd \bar z_i$. Then we can write
	\begin{equation*}
		 \pp \pi=f_i \dd \bar z_i+h(\frac{\dd \bar w}{\bar w}-\pp_{z_i} \psi \dd \bar z_i), 
	\end{equation*} where $f_i$ and $h$ are sections of $\operatorname {End}(E)$.  Notice that $d\bar z_i$ is perpendicular to the $\frac{\dd \bar w}{\bar w}-\pp_{z_i}  \psi\dd \bar z_i$.
Since $\pp\pi$ is in $L^2$ with respect to $\omega$, by (\ref{estimate for 1-forms}) we know that 
	\begin{equation}\label{integrability}
		\mathlarger\int \left(|f_i|^2(-\log |w|)^{1-a}+|h|^2(-\log |w|)^{2-a}\right)\frac{(-\log |w|)^{(n-1)(a-1)+a-2}}{|w|^2}d\lambda< \infty.
	\end{equation}
	Then we know that $f_i-h\pp_{z_i}\psi,\ \frac{h}{\bar w }$ are all $L^2$-integrable with respect to the Lebesgue measure. Therefore the claim is proved: $$\pp\pi=(f_i-h\pp_{z_i}\psi)\dd \bar z_i+\frac{h}{\bar w}\dd \bar w\in L^2(\olsi X;\omega_{\olsi X},\olsi H_0).$$
	Then Uhlenbeck-Yau's result (Theorem \ref{uhlenbeckyau}) implies that there exists a coherent subsheaf $\olsi S$ of $\olsi E$ such that $\pi=\pi^{\olsi H_0}_{\over S}$ outside the singular set of $\olsi S$. Taking the double dual, we may assume $\olsi S$ is reflexive. By the integrability condition (\ref{integrability}), $\pp \pi_{\over S}^{\olsi H_0}|_D=0$, which means that $\olsi S|_D$ is a splitting factor of $\olsi E|_D$ since $\olsi E|_D$ is polystable.	
\end{proof}

Now we are ready to prove the main theorem. We decompose it into two propositions.
\begin{proposition}\label{only if part}
	Suppose there exists an $\omega$-PHYM metric $H$ in $\mathcal P_{H_0}$, then $\olsi{E}$ is $\left(c_1(D),[\omega_0]\right)$-stable.
\end{proposition}
	
\begin{proof}
	Suppose there is a reflexive subsheaf $\olsi S$ of $\olsi E$ with $0< \rank(\olsi S)< \rank (\olsi E)$ such that $\mu(\olsi S,c_1(D))\geq c_1(\olsi E,c_1(D))$, we need to show that $\mu(\olsi S, [\omega_0])<\mu(\olsi E, [\omega_0])$. By \cite{mistretta} for any coherent refelxive sheaf $\mathcal E$ on $\olsi X$, we have 	
	\begin{equation*}
		\mu(\mathcal E,c_1(D))=\mu(\mathcal E|_D,c_1(D)|_D)=\mu(\mathcal E|_D,c_1(N_D)).
	\end{equation*} 
	(When $\mathcal E$ is a vector bundle, this follows from the fact that the first Chern class $c_1(D)$ is the Poincar\'e dual of the homology class defined by the divisor $D$. For a general reflexive sheaf, the key point is to show that $c_1(\mathcal E)|_D=c_1(\mathcal E|_D)$ using the fact that $\mathcal E$ is locally free outside an analytic set of (complex) codimension at least 3.)
	Therefore we have
	\begin{equation}\label{degree exceeding}
			\mu(\olsi S|_D,c_1(N_D))\geq \mu(\olsi E|_D,c_1(N_D)).
	\end{equation}
 By assumption, $\olsi E|_D$ is $c_1(N_D)$-polystable, so (\ref{degree exceeding}) implies that $\olsi S|_D$ is locally free and is a splitting factor of $\olsi E|_D$. Then by Lemma \ref{key observation}, we have 
	\begin{equation*}
		\pp\pi_S^{H_0}\in L^2(X,\omega, H_0).
	\end{equation*}

	\begin{claim}
		$\pp\pi_S^H\in L^2(X,\omega,H_0)=L^2(X,\omega,H)$.
	\end{claim}   For simplicity of notation, in the following we omit the dependence on $S$. By the definition of $\mathcal P_{H_0}$ and $H\in \mathcal P_{H_0}$, we know that $H=H_0 e^s$ with $\left\Vert s\right\Vert_{L^{\infty}}+\left\Vert \pp s\right\Vert_{L^2}<\infty$. The claim follows directly from the following pointwise inequality (outside the singular set $\Sigma$ of $S$) 
	\begin{equation}\label{projection with respect to different metrics}
		|\pp\pi^H|\leq C \left(|\pp s|+|\pp \pi^{H_0}|\right),
	\end{equation} where $C$ is a constant independent of points and all the norms are with respect to $H_0$. Let $r_0$, $r$ denote the rank of $S$ and $E$ respectively. Near any given point $p\in X\backslash \Sigma$, we can find a local holomorphic basis $\{e_1,\cdots,e_{r_0},e_{r_0+1},\cdots,e_r\}$ of $E$ such that

	\begin{equation*}
	\begin{aligned}
		&S=Span\{e_1,\cdots,e_{r_0}\},\\
&\left<e_i,e_j\right>_{H_0}(p)=\delta_{ij},\\
&\pp \left<e_i,e_j\right>_{H_0}(p)=0\  \text{for $1\leq i,j\leq r_0$ and $r_0+1\leq i,j\leq r$ }.
	\end{aligned}
	\end{equation*}
In the following we use Einstein summation convention and use $i,j$ to denote numbers from 1 to $r$, $\alpha,\beta$ to denote numbers from 1 to $r_0$. Under this basis $\pi^{H_0}$ can be written as $$e_{\alpha}^{\vee}\otimes e_{\alpha}+H_{0,i\beta}\widetilde H_0^{\beta\alpha}e_i^{\vee}\otimes e_{\alpha},$$ where we view $H_0=(H_{0,ij})=(\left<e_i,e_j\right>_{H_0})$ as a matrix and $\widetilde H_0=(\widetilde H_{0,\alpha\beta})=(\left<e_{\alpha},e_{\beta}\right>_{H_0})$ as a submatrix of $H_0$. Then \begin{equation*}
	|\pp\pi^{H_0}|(p)=\sum_{i,\alpha}|\pp H_{0,i\alpha}|(p).
\end{equation*}
	Similarly, $\pi^H$ can be written as $e_{\alpha}^{\vee}\otimes e_{\alpha}+H_{i\beta}\widetilde H^{\beta\alpha}e_i^{\vee}\otimes e_{\alpha}$. Note that as a matrix $H=H_0h$, where $h$ is the matrix representation of $e^s$ under the basis $\{e_i\}_{i=1}^r$. Since $\left\Vert s\right\Vert_{L^{\infty}}<\infty$, we have 
	\begin{equation*}
	\begin{aligned}
		|\pp\pi^H|(p)&\leq C\left(\sum|\pp H_{i\alpha}|(p)+|\pp h_{ij}|(p))\right)\\
		&\leq C\left(\sum |\pp H_{0,i\alpha}|(p)+|\pp h_{ij}|(p)\right)\\
		&\leq C\left(|\pp\pi^{H_0}|(p)+|\pp s|(p)\right), 
	\end{aligned}
	\end{equation*}which gives (\ref{projection with respect to different metrics}).

Let $\pi=\pi^H_S$. Using the Chern-Weil formula and the fact that $H\in \mathcal P_{H_0}$ is PHYM, we have 
\begin{equation*}
\begin{aligned}
	 \Lambda_{\omega}\tr(F_{S,H})&=\Lambda_{\omega}\tr(F_{E,H}\circ\pi)-|\pp\pi|^2\\
	 &=\frac{\rank (S)}{\rank (E)}\Lambda_{\omega}\tr(F_{E,H_0})-|\pp\pi|^2,
\end{aligned}
\end{equation*} 
and consequently is $L^1$. 

\begin{claim}
	\begin{equation*}
	\frac{1}{\rank (S)}\int_X\tr(F_{S,H})\wedge \omega^{n-1}=\mu(\olsi S,{[\omega_0]}).
\end{equation*}
\end{claim}

Assume this for a moment, then by Lemma \ref{degree coincide}, we know that 
\begin{equation*}
	\mu(\olsi S,[\omega_0])\leq \mu(\olsi E,[\omega_0]),
\end{equation*}
and equality holds if and only if $\pp\pi =0$. Suppose $\pp \pi=0$. Since $\left|\pi\right|_{\olsi H_0}=\left|\pi\right|_{H_0}\leq C\left|\pi\right|_H\leq C$. Again by  \cite[Lemma 7.3]{demailly-big-book}, there is a global holomorphic section of $\operatorname {End}(\olsi E)$, which is still denoted by $\pi$, such that $\pi=\pi^H_S$ a.e. and  $\pi^2=\pi$. Note that since $\rank (\pi)=\tr(\pi)$ is real valued and holomorphic, it follows that $\rank \pi$ is a constant. Thus $\olsi E$ holomorphically splits as the direct sum of $\ker \pi$ and $\Im \pi$, which contradict with our assumption that $\olsi E$ is irreducible. Therefore we prove that
	\begin{equation*}
	\mu(\olsi S,{[\omega_0]})< \mu(\olsi E,{[\omega_0]}).
\end{equation*}

\noindent \textit{Proof of the} \textbf{claim}: since $H\in \mathcal P_{H_0}$, we have 
\begin{equation*}
	\tr(F_{S,H})-\tr(F_{S,H_0})=\partial \pp u,
\end{equation*}
for a bounded real valued smooth function $u$ with $|\nabla u|\in L^2$. By Lemma \ref{exact integral imply degree 0},
\begin{equation*}
	\int\tr(F_{S,H})\wedge\omega^{n-1}=\int\tr(F_{S,H_0})\wedge\omega^{n-1}.
\end{equation*}
By the same argument in Lemma \ref{degree coincide}, we can show 
\begin{equation*}
	\int\tr(F_{S,H_0})\wedge\omega^{n-1}=\mu(\olsi S,{[\omega_0]}).
\end{equation*}
Hence we complete the proof of the claim.
\end{proof}

\begin{proposition}\label{existence result}
	Let $H_0$ be the metric constructed in Lemma \ref{good initial metric}. Suppose $\olsi E$ is $\left(c_1(D),{[\omega_0]}\right)$-stable, then there exists a unique $\omega$-PHYM metric $H$ in $\mathcal P_{H_0}$.
\end{proposition}

\begin{proof}
	Uniqueness is obvious. Suppose we have two $\omega$-PHYM metrics $H_1$, $H_2\in \mathcal P_{H_0}$. Let $h=H_1^{-1}H_2$. By the definition of $\mathcal P_{H_0}$, we know that $\det h=1$ and  $h$ is both bounded from above and below and $|\pp h| \in L^2$. Then by taking the trace of the differential equality in Lemma \ref{basic differential inequlaties}-(2), we get
	\begin{equation*}
		\Delta_{\omega}\tr(h)=-|\pp hh^{-\frac{1}{2}}|^2.
	\end{equation*}
By Lemma \ref{exact integral imply degree 0}, 
$$\int|\pp h h^{-\frac{1}{2}}|^2=-\int\Delta_{\omega} \tr(h)=0.$$ Therefore $\pp h=0$ and since $h$ is self-adjoint with respect to $H_i$, it is parallel with respect to the Chern connection determined by $(\pp,H_i)$. Then its eigenspaces give a holomorphic decomposition of $\olsi E$ which contradicts the assumption that $\olsi E$ is irrducible unless $h$ is a multiple of identity map. Since $\det h=1$, it must be that $h$ is the identity map, i.e. $H_1=H_2$.
	
	For the existence part, we follow Simpson and Mochizuki's argument \cite{simpson,mochizuki}. For completeness, we include some details. Let $\{X_i\}$ be an exhaustion of $X$ by compact domains with smooth boundary and we solve Dirichlet problems on every $X_i$ using Donaldson's theorem (Theorem \ref{donaldson dirichlet}). Then we have a sequence of PHYM metrics $H_i$ on $E|_{X_i}$ such that $H_i|_{\partial X_i}=H_0|_{\partial X_i}$ and $\det H_i=\det H_0$. Let $s_i$ be the endomorphism determined by $H_i=H_0h_i=H_0e^{s_i}$. Then we have $s_i|_{\partial X_i}=0$ and $\tr(s_i)=0$.
	
	 We argue by contradiction to prove a uniform $C^0$-estimate for $s_i$. First note that by Lemma \ref{basic differential inequlaties},
 $e^{s_i}$ satisfies the elliptic differential inequality 
	 \begin{equation}
	 	 \Delta\log(\tr(e^{s_i}))\leq |\Lambda F_{H_0}^{\perp}|
	 \end{equation}	 
	 therefore $\tr(e^{s_i})$ satisfies the weighted mean value inequality in  Lemma \ref{weighted sobolev inequality}. Since $\tr(e^{s_i})$ and $|s_i|$ are mutually bounded, we know that $|s_i|$ also satisfies the weighted mean value inequality (\ref{mean value inequality}). Lemma \ref{weighted sobolev inequality} plays an essential role since it ensures that after normalization we can have a nontrivial limit in $W^{1,2}_{loc}$.
	Suppose there is a sequence $s_i$ such that $\sup_{X_i}| s_i|_{H_0}\rightarrow \infty$ as $i\rightarrow \infty$. Then by Lemma \ref{weighted sobolev inequality}, we obtain $$l_i=\int_{X_i}|s_i|_{H_0}(1+r)^{-N_0}\rightarrow\infty.$$ Let $u_i=l_i^{-1}s_i$. Then by Lemma \ref{weighted sobolev inequality} again we obtain there is a constant $C$ independent of $i$ such that 
	\begin{equation}\label{nonzero condition}
		\text{$\int_{X_i}|u_i|(1+r)^{-N_0}=1$ and $|u_i|\leq C$},
	\end{equation}
	where the norms are with respect to the back ground metric $H_0$. Then following Simpson's argument, we can show that 
	  
\begin{lemma}
	After passing to a subsequence, $u_i$ converge weakly in $W^{1,2}_{loc}$ to a nonzero limit $u_{\infty}$. The limit $u_{\infty}$ satisfies the following property: if $\Phi: \mathbb{R} \times \mathbb{R} \rightarrow \mathbb{R}$ is a positive smooth function such that $\Phi\left(\lambda_{1}, \lambda_{2}\right)<\left(\lambda_{1}-\lambda_{2}\right)^{-1}$ whenever $\lambda_{1}>\lambda_{2}$, then
\begin{equation}\label{inequality for limits}
	\sqrt{-1} \int_{X} \tr(u_{\infty} \Lambda F_{H_0})+\int_{X}\left<\Phi\left(u_{\infty}\right)\left(\pp  u_{\infty}\right), \pp u_{\infty}\right>_{H_0} \leq 0.
\end{equation}

\end{lemma}    
	  \begin{proof}By Theorem \ref{phym metrics minimizing donaldson functional}
	  \begin{equation}\label{inequality of X_i}
	  	 \sqrt{-1} \int_{X_i} \tr(u_{i} \Lambda F_{H_0})+ l_{i} \int_{X_i}\left<\Psi\left(l_{i} u_{i}\right)\left(\pp  u_{i}\right), \pp u_{i}\right>_{H_0} \leq 0. 
	  	 \end{equation}
	  By the definition of $\Psi$ in (\ref{definition of Phi}), we know that as $l \rightarrow \infty, l \Psi\left(l \lambda_{1}, l \lambda_{2}\right)$ increases monotonically to $\left(\lambda_{1}-\lambda_{2}\right)^{-1}$ if $\lambda_{1}>\lambda_{2}$ and $\infty$ if $\lambda_{1} \leq \lambda_{2}$. Fix a $\Phi$ as in the statement of the lemma. We know that for all $A>0$ there exists $l_A$ such that if $|\lambda_i|\leq A$ and $l>l_A$, then we have
	  \begin{equation}\label{property for functions}
	  	\Phi\left(\lambda_{1}, \lambda_{2}\right)<l \Psi\left(l \lambda_{1}, l \lambda_{2}\right).
	  \end{equation}   
	    Since sup $\left|u_{i}\right|$ are bounded, its eigenvalues are also bounded. Then by (\ref{inequality of X_i}) and (\ref{property for functions}), we obtain that for $i$ sufficiently large
	       \begin{equation}\label{inequality for u_i}
	    	  	 \sqrt{-1} \int_{X_i} \tr(u_{i} \Lambda F_{H_0})+  \int_{X_i}\left<\Phi\left(u_{i}\right)\left(\pp  u_{i}\right), \pp u_{i}\right>_{H_0} \leq 0.	 
	    \end{equation}

Again since $\sup|u_i|$ is bounded we can find $\Phi$ satisfying the assumption in the lemma and $\Phi(u_i)=c_0$ for all $i$, where $c_0$ a fixed small positive number. Then by (\ref{inequality for u_i}) and the construction of $H_0$, there exists a positive constant $C$ such that  
$$\int_{X_i}|\pp u_i|^2\leq C.$$ Therefore by a diagonal sequence argument and after passing to a subsequence we may assume $u_i$ converge weakly in $W^{1,2}_{loc}$ to a limits $u_{\infty}$ with $\int_X|\pp u_{\infty}|^2\leq C$. We claim that $u_{\infty}\neq 0$. Indeed by (\ref{nonzero condition}), there exists a compact set $K\subseteq X$ independent of $i$ such that 
\begin{equation*}
	\int_{K}|u_i|(1+r)^{-N}\geq \frac{1}{2}.
\end{equation*}
Since on compact sets the embedding from $W^{1,2}$ to $L^1$ is compact, after taking the limit, we get $\int_K|u_{\infty}|(1+r)^{-N}\geq \frac{1}{2}$. In particular $u_{\infty}\neq 0$.

Next we prove (\ref{inequality for limits}). By the uniform boundedness of $u_i$, the $O(r^{-N_0})$ decay property of $|\Lambda F_{H_0}|$ and the nonnegativity of  the second term of the left hand side in (\ref{inequality for u_i}), we know that there exists $\epsilon_i\rightarrow 0$ such that for any $j\geq i$, we have
\begin{equation*}
	 \sqrt{-1} \int_{X_i} \tr(u_{j} \Lambda F_{H_0})+  \int_{X_i}\left<\Phi\left(u_{j}\right)\left(\pp  u_{j}\right), \pp u_{j}\right>_{H_0} \leq \epsilon_i.	 
\end{equation*}
Note that $\left<\Phi\left(u_{j}\right)\left(\pp  u_{j}\right), \pp u_{j}\right>_{H_0}=|\Phi^{\frac{1}{2}}(u_j)(\pp u_j)|_{H_0}^2.$ By \cite[Proposition 4.1]{simpson}, we know that on each fixed $X_i$, $\Phi^{\frac{1}{2}}(u_j)\rightarrow \Phi^{\frac{1}{2}}(u_{\infty})$ in $\operatorname{Hom}\left(L^{2}, L^{q}\right) \text { for any } q<2$.
Since $\pp u_j$ converge weakly in $L^2(X_i)$ to $\pp u_{\infty}$, we obtain that $\Phi^{\frac{1}{2}}(u_j)(\pp u_j)$ converge weakly to $\Phi^{\frac{1}{2}}(u_{\infty})(\pp u_{\infty})$ in $L^q(X_i)$ for any $q<2$. Then we know that for any $q<2$, 
\begin{equation*}
\begin{aligned}
	\left\Vert\Phi^{\frac{1}{2}}(u_{\infty})(u_{\infty})\right\Vert_{L^q(X_i)}^2&\leq \liminf_{j\rightarrow\infty}\left\Vert\Phi^{\frac{1}{2}}(u_{j})(u_{j})\right\Vert_{L^q(X_i)}^2\\
&\leq\operatorname{Vol}(X_i)^{\frac{2}{q}-1}\liminf_{j\rightarrow \infty}\left\Vert\Phi^{\frac{1}{2}}(u_{j})(u_{j})\right\Vert_{L^2(X_i)}^2\\
&\leq\operatorname{Vol}(X_i)^{\frac{2}{q}-1}\left(\epsilon_i-\lim_{j\rightarrow \infty}\ii\int_{X_i}\tr(u_j\Lambda F_{H_0})\right)\\
&\leq \operatorname{Vol}(X_i)^{\frac{2}{q}-1}\left(\epsilon_i-\ii\int_{X_i}\tr(u_{\infty}\Lambda F_{H_0})\right).
\end{aligned}
	\end{equation*}
Let $q\rightarrow2$, we obtain
\begin{equation*}
	\ii\int_{X_i}\tr(u_{\infty}\Lambda F_{H_0})+\left\Vert\Phi^{\frac{1}{2}}(u_{\infty})(u_{\infty})\right\Vert_{L^2(X_i)}^2\leq \epsilon_i.
\end{equation*}
	 Letting $i\rightarrow\infty$, the inequality (\ref{inequality for limits}) is proved.	
\end{proof}

Simpson's argument in \cite[Lemma 5.5 and Lemma 5.6]{simpson} can be applied verbatim to the infinite volume case, so we have
\begin{lemma}[{\cite{simpson}}] Let $u_{\infty}$ be a limit obtained in the previous lemma. Then we have 
\begin{itemize}
	\item [(1)]The eigenvalues of $u_{\infty}$ are constant and not all equal.
	\item [(2)]Let $\Phi: \mathbb{R} \times \mathbb{R} \longrightarrow(0, \infty)$ be a $C^{\infty}$-function such that $\Phi\left(\lambda_{i}, \lambda_{j}\right)=0$ if $\lambda_{i}>\lambda_{j}$. Then $\Phi\left(u_{\infty}\right)\left(\bar{\partial} u_{\infty}\right)=0$.
\end{itemize}
\end{lemma}

Let $\lambda_1 \leq \lambda_2\leq \cdots\leq \lambda_{\rank(E)}$ denote the eigenvalues of $u_{\infty}$. Let $\gamma$ be an open interval between the eigenvalues (since eigenvalues of $u_{\infty}$ are not all equal by the previous lemma, there exists such a nonempty interval). We choose a $C^{\infty}$-function $p_{\gamma}: \mathbb{R} \longrightarrow(0, \infty)$ such that $p_{\gamma}\left(\lambda_{i}\right)=1$ if $\lambda_{i}<\gamma$, and $p_{\gamma}\left(\lambda_{i}\right)=0$ if $\lambda_{i}>\gamma$. Set $\pi_{\gamma}:=p_{\gamma}\left(u_{\infty}\right)$, see Section \ref{definition of new endmorphisms} for the definition. Then one can easily show that \cite{simpson, mochizuki}
	\begin{itemize}
		\item [(1)] $\pi_{\gamma}^2=\pi_{\gamma}$, $(\operatorname{id}-\pi_{\gamma})\circ\pi_{\gamma} =0$ and $\pi_{\gamma}$ is self-adjoint with respect to $H_0$.
		\item [(2)] $\int_X|\pp\pi_{\gamma}|^2<\infty$.
\end{itemize}
Moreover using (\ref{inequality for limits}), Simpson proved that 
\begin{lemma}[{\cite{simpson}}]\label{degree exceed}
	There exists at least one $\gamma$ such that 
	\begin{equation*}
		\frac{1}{\tr(\pi_{\gamma})}\left(\ii\int_X\tr(\pi_{\gamma}\Lambda F_{H_0})-\int_X|\pp\pi_{\gamma}|^2\right)\geq \frac{1}{\rank(E)}\ii\int_X\tr(\Lambda F_{H_0}).
	\end{equation*}

\end{lemma}

	 By Lemma \ref{key observation}, we get a filtration of $\olsi E$ by coherent reflexive subsheaves $\olsi S_i$ whose restrictions to $D$ are splitting factors of $\olsi E|_D$. Since we assume that $\olsi{E}|_D$ is $c_1(N_D)$-polystable, we know that for every $i$ 
	 \begin{equation*}
	 	\mu(\olsi S_i|_D,c_1(N_D))=\mu(\olsi E|_D, c_1(N_D)).
	 \end{equation*}
	 Then again by \cite{mistretta}, we have 
	\begin{equation}\label{first degree coincides}
		\mu(\olsi S_i,c_1(D))=\mu(\olsi E, c_1(D)).
	\end{equation}
	Note that Lemma \ref{degree exceed} is equivalent to the statement that there exists at least one $\olsi S_i$ such that 
	\begin{equation*}
		\int_X\tr(F_{S_i,H_0})\wedge\omega^{n-1}\geq \int_X\tr(F_{E,H_0})\wedge\omega^{n-1}.	\end{equation*}
Then by Lemma \ref{degree coincide}, 
	\begin{equation}
			\mu(\olsi S_i,{[\omega_0]})\geq \mu(\olsi E,{[\omega_0]}).
	\end{equation}
which contradicts with the $(c_1(D),{[\omega_0]})$-stability assumption. Therefore we do have a uniform $C^0$-estimate for $s_i$. 
	
	Bando-Siu's interior regularity result Theorem \ref{bando-siu interior estimate} can be applied to get local uniform estimate for all derivatives of $s_i$. Then we can take limits to get a smooth section $s\in \operatorname{End}( E)$, which is self-adjoint with respect to $H_0$ and $\tr(s)=0$ and more importantly 
	\begin{equation*}
		\Vert s\Vert_{L^{\infty}}<\infty\ \text{and $H=H_0\rm e$$^{s}$ is a PHYM metric}.
	\end{equation*}
	Then we use Mochizuki's argument in \cite[Section 2.8]{mochizuki} to show that 
	\begin{equation*}
		|\pp s|\in L^2(X,\omega, H_0).
	\end{equation*}
	 Indeed taking the trace of the equality in Lemma \ref{basic differential inequlaties}-(2) and noting that $H_i=H_0h_i$ is PHYM, we have 
	\begin{equation}\label{to show L^2}
		\Delta \tr(h_i)=-\tr(h_i\ii\Lambda F_{H_0}^{\perp})-|h_i^{\frac{1}{2}}\pp(h_i)|^2.
	\end{equation}
	Since $\det h_i=1$ and $h_i|_{\partial X_i}=\operatorname{id}$, we know that $\nabla_{\nu_i}\tr(h_i)\leq 0$, where $\nu_i$ denotes the outward unit normal vector of $\partial X_i$. Integrating (\ref{to show L^2}) over $X_i$ and using Stoke's theorem in the left hand side, we obtain
	\begin{equation*}
		\int_{X_i}|h_i^{\frac{1}{2}}\pp(h_i)|^2\leq -\int_{X_i}\tr(h_i\ii\Lambda F_{H_0}^{\perp}).
	\end{equation*}
	Since we have uniform $C^0$-estimate for $s_i=\log h_i$, there exist constants $C_1$ and $C_2$ independent of $i$ such that 
	\begin{equation*}
		\int_{X_i}|\pp s_i|^2\leq C_1\int_{X_i}|\pp h_i|^2\leq C_2
	\end{equation*}
	Let $i\rightarrow \infty$, we have $\int_X|\pp s|^2\leq C_2$.
	\end{proof}

\noindent\textit{On the stability condition.} Note that global semistability is known \cite{mistretta}, if we assume the restriction to $D$ is semistable. There do exist irreducible holomorphic vector bundles which are polystable when restricted to $D$ but not globally stable, even under more restrictive assumptions that $\olsi{X}$ is Fano and $D\in |K_{\olsi{X}}^{-1}|$.
%\begin{theorem}Let $M$ be a 2-dimensional Fano manifold, $D\in |K^{-1}_{M}|$ be a smooth anticanonical divisor. Suppose we have an exact sequence \begin{equation*}
	%	0\rightarrow S\rightarrow E\rightarrow Q\rightarrow 0
	%\end{equation*} of holomorphic vector bundles on $M$ and $E|_D$  is $c_1(N_D)$-polystable and splits (holomorphically) as the direct sum of $S|_D$ and $Q|_D$. Then globally on $M$, $E$ is holomorpically isomorphic to $S\oplus Q$.
%\end{theorem}
\begin{example}\label{example of nonsplitting}
 Recall that for holomorphic vector bundles $S$, $Q$ over a complex manifold $M$, all exact sequences $0\rightarrow S\rightarrow E\rightarrow Q\rightarrow 0$ of holomorphic vector bundles are classified by elements $\beta\in H^1(M,\operatorname {Hom}(Q,S))$ and in particular the exact sequence splits holomorphically if and only if the corresponding element $\beta=0$. Now taking $M$ to be $\mathbb {CP}^1\times \mathbb {CP}^1$, $D$ to be a smooth anticanonical divisor. Then $c_1(D)$ is a K\"ahler class and $D$ itself is an elliptic curve. Choose $\operatorname {Hom}(Q,S)=L=p_1^*(\mathcal O(2))\otimes p_2^*(\mathcal O(-2))$. Then by the K\"unneth's formula, $\dim H^1(M,L)=3$. Note that $\deg (L|_D)=0$, which by Serre duality implies $\dim H^1(D,L)=\dim H^0(D,L^*)\leq 1$. So there exists a  class $\beta \in H^1(M,L)$ corresponding to a non-splitting exact sequence of holomorphic vector bundles whose restriction to $D$ splits as a direct sum of two line bundles with the same degree. Therefore $E$ itself is not $c_1(D)$-stable but $E|_D$ is $c_1(N_D)$-polystable. Such an $E$ is irreducible, because if $E=L_1 \oplus L_2$,  then $\deg (L_i,c_1(D))=\deg (L_i|_D)=0$ since $E|_D$ is polystable of degree 0, which implies that $S$ has to be one of the $L_i$ and $Q$ is the other one. This contradicts with the construction of $E$.
\end{example}

%%%%%%%%%%%%%%%%%%%%%%%%%%%%%%%%%%%%%%%%%%%%%%%%%%%%%%%
%%%%%%%%%%%%%%%%%%%%%%%%%%%%%%%%%%%%%%%%%%%%%%%%%%%%%%%
%%%%%%%%%%%%%%%%%%%%%%%%%%%%%%%%%%%%%%%%%%%%%%%%%%%%%%%

\section{Discussion} \label{miscellaneous discussion}

\subsection{More results on the existence of PHYM metrics}\label{relation with some results} 
By Donaldson's theorem on the solvability of Dirichlet problem (Theorem \ref{donaldson dirichlet}), the elliptic differential inequality (Lemma \ref{basic differential inequlaties}-(3)), the maximal principle and Bando-Siu's interior estimate (Theorem \ref{bando-siu interior estimate}), we get the following well-known existence result.

\begin{theorem}
	
\label{general existence result}Let $(M,\omega,g)$ be a complete K\"ahler manifold, $E$ be a holomorphic vector bundle on $M$. 
	Suppose there exists a smooth hermitian metric $H_0$ on $E$ such that the equation 
	\begin{equation}\label{laplace equaiton}
		\Delta u=|\Lambda F_{H_0}^{\perp}|
	\end{equation}
admits a positive solution $u$. Then there exists a  smooth hermitian metric $H=H_{0}e^s$ satisfying 
\begin{equation*}
	\tr(s)=0,\ |s|_{H_0}\leq C_1 u \text{ and } \Lambda F_H^{\perp}=0.
\end{equation*}
Moreover, if $u$ is bounded and $|\Lambda F_{H_0}|\in L^1$ then $|\pp s|\in L^2$.
\end{theorem}

There are many examples for which (\ref{laplace equaiton}) has a positive solution and even bounded solutions \cite{bando,ni,ni-shi-tam}. 
\begin{itemize}
	\item [(1)] Suppose $(M,g)$ is asymptotically conical and $|\Lambda F_{H_0}^{\perp}|=O(r^{-2-\epsilon})$ for some $\epsilon>0$, then (\ref{laplace equaiton}) admits a solution $u$ with $|u|=O(r^{-\epsilon})$.
	\item [(2)] Suppose $(M,g)$ is non-parabolic (i.e admits a positive Green's function) and $|\Lambda F_{H_0}^{\perp}|\in L^1$, then (\ref{laplace equaiton}) admits a positive solution.
	\item [(3)] Suppose $(M,g)$ has nonnegative Ricci curvature, $|\Lambda F_{H_0}^{\perp}|=O(r^{-2})$ and 
	\begin{equation*}
		\frac{1}{\operatorname{Vol}(B_r)}\int_{B_r} |\Lambda F_{H_0}^{\perp}|=O(r^{-2-\epsilon}),
	\end{equation*}
for some $\epsilon>0$, then (\ref{laplace equaiton}) admits a bounded solution. In particular, if $(M,g)$ has nonnegative Ricci curvature, volume growth order greater than 2, $|\Lambda F_{H_0}^{\perp}|=O(r^{-2})$ and $|\Lambda F_{H_0}^{\perp}|\in L^1$, then (\ref{laplace equaiton}) admits a bounded solution.
\end{itemize}

Theorem \ref{general existence result} can not be applied to $(X,\omega,g)$ satisfying \textit{Assumption 1} since we do not know whether (\ref{laplace equaiton}) admits a positive solution (for this volume growth order at most 2 is a key issue). And actually Theorem \ref{main theorem} tells us that there are some obstructions for the existence of $\omega$-PHYM   metrics which are mutually bounded with the initial metric. 

Such a phenomenon also appears when we seek a bounded solution for the Poisson equation 
\begin{equation}\label{poission equation}
	\Delta u=f
\end{equation}
on a complete noncompact Riemannian manifold $(M,g)$ with nonnegative Ricci curvature. Suppose $f$ is compactly supported for simplicity, then we know that 
\begin{itemize}
	\item [(1)] if the volume growth order is greater than 2, i.e. there is a constant $c>0$ such that $\operatorname {Vol}(B_r)\geq c r^{2+\epsilon}$ for some $\epsilon>0$, then (\ref{poission equation}) admits a bounded solution. (Since by Li-Yau \cite{li-yau}, $(M,g)$ admits a positive Green's function which is  $O(r^{-\epsilon})$ at infinity, a bounded solution of (\ref{poission equation}) is obtained by the convolution with the Green's function.)
	\item [(2)] if the volume growth order does not exceed 2, i.e. there is a constant $C>0$ such that $\operatorname {Vol}(B_r)\leq C (r+1)^{2}$, then (\ref{poission equation}) admits a bounded solution if and only if $\int_Mf=0$. (For the ``if'' direction, see \cite[Theorem 1.5]{hein2011}. For the ``only if'' direction, suppose we have a bounded function $u$ and a compactly supported function $f$ such that $\Delta u=f$. Then by Cheng-Yau's gradient estimate \cite{cheng-yau}, we obtain $|\nabla u|\leq \frac{C}{r}$ for some $C>0$ independent of $r$. Multiplying $u$ both sides in (\ref{poission equation}) and integrating by parts, we obtain that $|\nabla u|\in L^2$. Then Lemma \ref{exact integral imply degree 0} implies $\int_Mf=0$.)
\end{itemize}

Next we discuss another result whose proof is  similar to the proof of Theorem \ref{main theorem}.
Let $(\olsi{X},\ols{\omega})$ be an $n$-dimensional ($n\geq 2$) compact K\"ahler  manifold, $D$ be a smooth divisor. Let $\ols \omega_D=\ols{\omega}|_D$ denote the restriction of $\ols{\omega}$ to $D$ and $X=\olsi{X}\backslash D$ denote the complement of $D$ in $\olsi{X}$. Let $L_D$ be the line bundle determined by $D$ and $S\in H^0(\olsi{X},L_D)$ be a defining section of $D$. Fix a hermitian metric $h$ on $L_D$. Then after scaling $h$, the function $t=-\log |S|_h^2$ is smooth and positive on $X$.  For any smooth function $F:(0,\infty)\longrightarrow \mathbb R$ with $|F'(t)|\rightarrow 0$ as $t\rightarrow \infty$ and $F''(t)\geq0$ there exists a large constant $A$ such that 
\begin{equation}\label{kahler forms used}
	\omega=A\ols{\omega}+dd^cF(t)
\end{equation}
is a K\"ahler form on $X$. By scaling $\ols{\omega}$ we may assume $A=1$. One can easily check that $\omega$ is complete is and only if $\int_{1}^{\infty} \sqrt{F''}=\infty$ and it always has finite volume. In the following, we always assume the function $F$ satisfies $|F'(t)|\rightarrow 0$ as $t\rightarrow \infty$ and $F''(t)\geq0$. Then we can state  assumptions on $\omega$.

\begin{assumption}
	 Let $\omega$ be the K\"ahler form defined by (\ref{kahler forms used}) and $g$ be the corresponding Riemannian metric. We assume that
 \begin{itemize}
	\item [(1)] the sectional curvature of $g$ is bounded.
	\item [(2)] $C^{-1}t^{-2+\epsilon}\leq F''(t)\leq C$ for some constant $C,\epsilon>0$ and $t$ sufficiently large.
\end{itemize}

\end{assumption}

 A consequence of these assumptions is that $(X,g)$ is complete and of $(K,\alpha,\beta)$-polynomial growth defined in \cite[Definition 1.1]{tian-yau1}, so we can use the weighted Sobolev inequality as we did for the proof of Lemma \ref{weighted sobolev inequality}.

 Let $\olsi{E}$ be an irreducible holomorphic vector bundle on $\olsi{X}$ such that $\olsi{E}|_D$ is $\omega_D$-polystable. Then by Donaldson-Uhlenbeck-Yau theorem, there exists a hermitian metric $H_D$ on $\olsi{E}|_D$ such that 
 \begin{equation}\label{hym condition}
 	\Lambda_{\ols \omega_D}F_{H_D}^{\perp}=0.
 \end{equation}
Extend $H_D$ smoothly to get a smooth hermitian metric $H_0$ on $\olsi{E}$. Then by (\ref{hym condition}) and \textit{Assumption 2}-(2), one can easily show that

\begin{lemma}\label{good initial metric in general setting}
 There exists a $\delta>0$ such that 
	$|\Lambda_{\omega} F_{H_0}^{\perp}|=O(e^{-\delta t})$.
\end{lemma}

Then we have the following result
\begin{theorem}\label{theorem for general divisor}
	Suppose $(X,\omega)$ satisfies \textit{Assumption 2} and $\olsi{E}|_D$ is $\ols\omega_D$-polystable. Let $H_0$ be a hermitian metric as above and $\mathcal P_{H_0}$ be defined by (\ref{define of P_H}). Then there exists an $\omega$-PHYM metric in $\mathcal P_{H_0}$ if and only if $\olsi{E}$ is $\ols{\omega}$-stable.
\end{theorem}

Using the argument in Proposition \ref{only if part}, the ``only if'' direction follows from Lemma \ref{exact integral imply degree 0} and the following lemma . 

\begin{lemma}
	For every smooth closed (1,1)-form $\theta$ on $\olsi{X}$, we have
	\begin{equation}\label{degree equals 2}
		\int_X \theta\wedge \omega^{n-1}=\int_X\theta\wedge \ols{\omega}^{n-1}
	\end{equation}
\end{lemma}

\begin{proof}
	Firstly note that since there exists a positive number $c>0$ such that $\omega> c\ols{\omega}$ and $\int \omega^n<\infty$, the left hand side of (\ref{degree equals 2}) is well-defined.
	Therefore it suffices to show that for any $1\leq k\leq n-1$
	\begin{equation*}
	\int_X\theta\wedge \ols{\omega}^{n-1-k}\wedge (dd^cF)^{k}=0.
	\end{equation*}
	Let $S_{\epsilon}$ denote the level set $\{|S|_h=\epsilon\}$. By integration by part, it suffices to show that 
	\begin{equation}\label{to be proved for any k}
		\lim_{\epsilon\rightarrow 0}\int_{S_{\epsilon}}\theta\wedge \ols{\omega}^{n-1-k}\wedge (dd^cF)^{k-1}\wedge d^c F=0.
	\end{equation}
	
\noindent\textit{Case 1.} $k=1$. Note that with respect to the smooth back ground metric $\ols{\omega}$, $\operatorname{Vol}(S_{\epsilon})=O(\epsilon)$ and $|d^c F|\leq C|F'(t)|\epsilon^{-1}$ on $S_{\epsilon}$. Then (\ref{to be proved for any k}) follows from the assumption that $|F'|\rightarrow 0$ as $t\rightarrow\infty$.

\noindent\textit{Case 2.} $2\leq k\leq n-1$.  
Then (\ref{to be proved for any k}) follows from the fact that $|F'(t)|\rightarrow 0$ as $t\rightarrow \infty$ and $d^ct\wedge d^c t=0$.
\end{proof}

For the ``if'' direction, the argument in Proposition \ref{existence result} applies. We will not give the details and just point out the following two observations which make the argument work in this setting. The key points are
\begin{itemize}
	\item [(1)] \textit{Assumption 2} and Lemma \ref{good initial metric in general setting} ensure that we can apply the weighted mean value inequality proved in Lemma \ref{weighted sobolev inequality}. 
	\item [(2)] We have $L^2(X,\omega)\subset L^2(\olsi{X},\ols{\omega})$ since $\omega\geq c\ols{\omega}$ for some $c>0$, therefore by Uhlenbeck-Yau's theorem (Theorem \ref{uhlenbeckyau}) a weakly projection map $\pi$ of $E$ over $X$ with $|\pp \pi|\in L^2(X,\omega)$ defines a coherent torsion free sheaf $\olsi{S}$ of $\olsi{E}$. 

\end{itemize}

\subsection{Calabi-Yau metrics satisfying \textit{Assumption 1}}\label{examples}
As mentioned in the Introduction, there do exist interesting K\"ahler metrics satisfying the \textit{Assumption 1}, which contain Calabi-Yau metrics on the complement of an anticanonical divisor of a Fano manifold and its generalizations \cite{tian-yau1, hsvz1, hsvz2}. We will call them Tian-Yau metrics. Here we give a sketch for the construction of these Calabi-Yau metrics and refer to \cite{hsvz2}-Section 3 for more details.

Let $\olsi{X}$ be an $n$-dimensional ($n\geq 2$) projective manifold, $D\in |K^{-1}_{\olsi{X}}|$ be a smooth divisor and $X=\olsi{X}\backslash D$ be the complement of $D$ in $\olsi{X}$. Suppose  that the normal bundle of $D$ in $\olsi{X}$, $N_D=K^{-1}_{\olsi{X}}|_D$ is ample. Fixing a defining section $S\in H^0(\olsi X, K^{-1}_{\olsi X})$ of the divisor $D$ whose inverse can be viewed as a holomorphic volume form $\Omega_X$ on $X$ with a simple pole along $D$. Let $\Omega_D$ be the holomorphic volume form on $D$ given by the residue of $\Omega_X$ along $D$. Using Yau's theorem \cite{yau1978} , there is a hermitian metric $h_D$ on $K^{-1}_{\olsi X}|_D$ such that its curvature form is a Ricci-flat K\"ahler metric $\omega_D$ with $$\omega_D^{n-1}=(\ii)^{(n-1)^2}\Omega_D\wedge\ols \Omega_D$$ by rescaling $S$ if necessary. One can show that the hermitian metric $h_D$ extends to a global hermitian metric $h_{\olsi X}$ on $K^{-1}_{\olsi X}$ such that its curvature form is nonnegative and positive in a neighborhood of $D$.

%\begin{proof} Let $L$ denote the anticanonical line bundle $K^{-1}_{\olsi X}$. Firstly one can show that there is a hermitian metric  $h$ on $L$ such that its curvature form is nonnegative and positive in a neighborhood of $D$. This is proved by Conlon-Hein in \cite{conlon-hein2}-Lemma 2.3 using a result of Griffith in \cite{griffiths}-Section VIII.1. 
%
% Using this hermitian metric $h$, the problem is reduced to the following: 
% Let $\theta$ be a closed real nonnegative  (1,1)-form  on $\olsi X$ and positive in neighborhood of $D$. Then for any smooth function $f$ on $D$ with $\theta|_D+dd^cf>0$ on $D$, there exists a smooth extension $F$ of $f$ to $\olsi X$ such that $\theta +dd^cF\geq 0$ and positive in a neighborhood of $D$. This is proved in \cite{degenerate}-Proposition 8.8 for $\theta$ being a K\"ahler form. And the proof given there also works for $\theta$ being only nonnegaitve and positive in a neighborhood of $D$.
%\end{proof}

By glueing a smooth positive constant on a compact set, we  get a global positive  smooth function $z$ which is equal to $(-\log |S|^2_{h_{\olsi X}})^{\frac{1}{n}}$ outside a compact set. For any $A\in \mathbb R$, we denote $h_A=h_{\olsi X}e^{-A}$ and $v_A=\frac{n}{n+1}(-\log |S|_{h_A}^2)^{\frac{n+1}{n}}$, which is viewed as a smooth function defined outside a compact set on $X$.
%\begin{lemma}
	%For any $A\in \mathbb R$, there is a $v_A \in C^{\infty}(X)$ such that $\partial\pp v_A\geq 0$ and $v_A=\frac{1}{n+1}(-\log |S|_{h_A}^2)^{\frac{n+1}{n}}$ outside a compact set of $X$.
%\end{lemma} 
We denote by $H^2_{c,+}(X)$ the subset of $\Im(H^2_c(X,\mathbb R)\rightarrow H^2(X,\mathbb R))$ consisting of 
classes $\mathfrak k$ such that $\int_Y\mathfrak{k}^p>0$ for any compact analytic subset $Y$ of $X$ of pure dimension $p>0$.

Then Hein-Sun-Viaclovsky-Zhang proved the following result
\begin{theorem}\cite{hsvz2}\label{generalized tianyau}
For every class $\mathfrak{k} \in H_{c,+}^{2}(X)$, there is a unique K\"ahler metric $\omega\in \mathfrak k $ such that
	\begin{itemize}
	\item [(1)] $\omega^n=(\ii)^{n^2} \Omega_X\wedge\ols{\Omega}_X$, and 
		\item [(2)] $|\nabla^l_{\omega}(\omega-\ii\partial\pp v_A)|_{\omega}=O\left(e^{-\delta z^{\frac{n}{2}}}\right)$ for some $\delta,A>0$ and all $l\geq 0$.
	\end{itemize}
	\end{theorem}
	
And from the construction in \cite{hsvz2}-Section 3, we have the decomposition $\omega=\omega_0+dd^c\varphi$, where $\omega_0$ is a smooth (1,1)-form on $\olsi{X}$ vanishing when restricted to $D$. And by Theorem \ref{generalized tianyau} and the estimate in \cite[Proposition 3.4]{hsvz1}, one can directly check that these K\"ahler metrics satisfy \textit{Assumption 1}.

\begin{remark}
	It was proved in \cite{hsvz1} that Tian-Yau metrics $\omega_{TY}$ can be realized as the rescaled pointed Gromov-Hausdorff limits of a sequence of Calabi-Yau metrics $\omega_k$ on a $K3$ surface. We expect that $\omega_{TY}$-PHYM connections we obtained in this paper give models for the limits of $\omega_k$-HYM connections on the $K$3 surface.
\end{remark}

\subsection{On the ampleness assumption of the normal bundle $N_D$}\label{general normal bundles} In this subsection, we first explain why we assume the normal bundle of $D$ is ample and then discuss the case where the normal bundle is trivial on compact K\"ahler surfaces.

 Let us start with a question. Suppose we have a nontrivial nef class $\alpha$ in $H^{1,1}(\olsi{X})$ and $D\in \ols{X}$ a smooth divisor, when does 
\begin{equation}
	\mu(\mathcal E,\alpha)=\mu(\mathcal E|_D,\alpha|_D)
\end{equation}
hold for every coherent reflexive sheaf $\mathcal E$ on $\olsi{X}$ ? A sufficient condition is that 
\begin{equation*}
	\alpha^{n-2}\wedge(\alpha-c_1(D))=0 \text{ in $H^{n-1,n-1}(\olsi{X})$}.
\end{equation*}
In order to have the above equality, a natural (possibly the only reasonable) choice is that $\alpha=c_1(D)$.

 To make the argument in this paper work, we also need the following property: 
 \textit{if a vector bundle $F$ on $D$ is polystable with respect to $\alpha|_D$ and $S$ is a coherent subsheaf of $F$ with the same $\alpha|_D$-degree as $F$, then $S$ is a vector bundle and is a splitting factor of $F$}. (Note that this does not follow from the definition since $\alpha|_D$ may not be a K\"ahler class. For example if $\alpha|_D$ is 0, then definitely it does not satisfy this property.) In general in order to have this property, we need $\alpha|_D$ to be a K\"ahler class. This is one of the reasons why we assume that the normal bundle of $D$ is ample, i.e. $c_1(D)|_D$ is a K\"ahler class. Another reason is that by assuming $N_D$  is ample, on the punctured disc bundle $\mathcal C$ we have explicit exact K\"ahler forms, which give models of the K\"ahler forms on $X$.

However if $\olsi{X}$ is a compact complex surface, in which case the divisor $D$ now is a smooth Riemann surface, then the property mentioned above always holds. Note that on a Riemann surface $D$, the slope of a vector bundle is canonically defined and independent of the choice of  cohomology classes on $D$.
\begin{lemma}Let $\olsi{X}$ be a compact K\"ahler surface and $D$ be a smooth divisor. Suppose $\olsi{E}|_D$ is polystable.
	Let $\olsi{S}$ be a coherent reflexive subsheaf of $\olsi{E}$. Then
	$\mu(\olsi{S},c_1(D))=\mu(\olsi{E},c_1(D))$ if and only if $\olsi{S}|_D$ is a splitting factor of $\olsi{E}|_D$.
\end{lemma}

Using this, most of the arguments in Section \ref{proof of the main theorem} can be modified to work for divisors $D$ with $c_1(N_D)=0$ in complex dimension 2. In the following, we assume $c_1(N_D)=0$ in $H^2(D,\mathbb R)$. Then it is easy to see that $c_1(D)$ is nef and by the global $\partial\pp$-lemma on $D$, we know that there exists a hermitian metric $h_D$ on $N_D$ with vanishing curvature.  Let $L_D$ be the line bundle determined by $D$ and $S\in H^0(\olsi{X},L_D)$ be a defining section of $D$. Then we can extend $h_D$ smoothly to get a smooth hermitian metric $h$ on $L_D$ and after a rescaling, we may assume that $t=-\log |S|_h^2$ is positive on $X$. 
In this case, we can consider (at least) all  monomials potentials with degree greater than 1
 \begin{equation}
	\mathcal H:=\left\{F(t)=At^a: A>0 \text{ is a constant and } a>1\right\}.
\end{equation}

%Define $\mathcal C, \mathcal D$ and $t$ as in (\ref{definition of model space}) and (\ref{definition of t}). Let $\ols{\omega}$ be a closed (1,1) form on $\olsi{X}$ such that $[\ols{\omega}]\in H^{1,1}(\olsi{X})$ is nef and $\ols{\omega}|_D$ is a K\"ahler form. Then every smooth positive function $f$ defined on $(0,\infty)$  defines a K\"ahler form on $\mathcal C$
% \begin{equation}
%	\omega_f=p^*\omega_D+f(t)\ii\partial t\wedge\pp t=p^*\omega_D+\ii\partial\pp F(t),
%\end{equation}
%where $F''=f$.

\begin{assumption}
Let $\omega$ be a K\"ahler form on $X$ and $g$ be the corresponding Riemannian metric. We assume that  
\begin{itemize}
	\item [(1)] the sectional curvature of $g$ is bounded.
	\item [(2)] the form $\omega$ can be written as $\omega_0+\ii\partial\pp F(t)$ for some $F\in \mathcal H$, where $\omega_0$ is a smooth closed (1,1)-form on $\olsi{X}$. 
\end{itemize}
\end{assumption}

 Suppose $(X,\omega,g)$ satisfies \textit{Assumption 3}, then we have the following consequences: 	\begin{itemize}
 		\item  the Riemannian metric $g$ is complete and has volume growth order at most 2,
 		\item  $(X,g)$ is of $(K,2,\beta)$-polynomial growth as defined in \cite[Definition 1.1]{tian-yau1} for some positive constants $K$ and $\beta$.
 		\end{itemize}

  Let $\olsi{E}$ be an irreducible holomorphic vector bundle over $\olsi{X}$ such that $\olsi{E}|_D$ is polystable with degree 0. Then by Donaldson-Uhlenbeck-Yau theorem  (for Riemann surfaces this was first proved by Narasimhan and Seshadri \cite{narasimhan}), there exists a hermitian metric $H_D$ on $\olsi{E}|_D$ such that 
  \begin{equation*}
  	\Lambda_{\omega_D}F_{H_D}=0.
  \end{equation*}
Since $D$ is a Riemann surface, this is equivalent to say that $H_D$ gives a flat metric on $\olsi{E}|_D$, i.e.
\begin{equation}\label{PHYM condition}
	F_{H_D}=0.
\end{equation}
Extending $H_D$ smoothly to get a hermitian metric $H_0$ on $\olsi{E}$ then by (\ref{PHYM condition}) and the proof of Lemma \ref{good initial metric}, we know that $H_0$ is already a good initial metric in the following sense:
\begin{equation}\label{again good initial metric} 
		|F_{H_0}|=O(e^{-\delta t}).
	\end{equation}
Then we have the following result, whose proof is essentially the same as that for Theorem \ref{main theorem}. We just point out the difference.
\begin{theorem}\label{theorem for trivial normal bundle}
	Suppose $(X,\omega)$ satisfies \textit{Assumption 3} and $\olsi{E}|_D$ flat. Let $H_0$ be a hermitian metric as above and $\mathcal P_{H_0}$ be defined by (\ref{define of P_H}). Then there exists an $\omega$-PHYM metric in $\mathcal P_{H_0}$ if and only if $\olsi{E}$ is $\left(c_1(D),[\omega_0]\right)$-stable.
\end{theorem}

 The argument in Section \ref{proof of the main theorem} can be applied if Lemma \ref{degree coincide} still holds. The analog of Lemma \ref{degree coincide} in this case is the following lemma, for which we need to assume $\olsi{E}|_D$ is flat.

\begin{lemma}Suppose $(X,\omega)$ satisfies \textit{Assumption 3} and $\olsi{E}|_D$ is flat. Let $H_0$ be a hermitian metric as above. Then we have the following equality:
	\begin{equation*}
		\int_{X}\frac{\ii}{2\pi}\tr(F_{H_0})\wedge \omega=\int_{\olsi X}c_1(\olsi E)\wedge [\omega_0]
		\end{equation*}
\end{lemma}

\begin{proof}
	By Chern-Weil theory, it suffices to show that 
	\begin{equation}\label{again degree vanish}
		\int_{X}\frac{\ii}{2\pi}\tr(F_{H_0})\wedge dd^c\varphi=0.
	\end{equation}
	 The argument in Lemma \ref{a good 1-form potential} can be used again to show that there exists a cut-off function $\chi$ supported on a compact set and a smooth 1-form $\psi$ supported outside a compact set such that
	\begin{equation*}
		dd^c \varphi=d d^c(\chi\varphi)+d\psi.
	\end{equation*} Moreover $|\psi|$ grows at most in a polynomial rate of $r$.
	 Then (\ref{again degree vanish}) follows from integration by parts and  \eqref{again good initial metric}.
\end{proof}

\begin{example}
	Let $\olsi{X}=\mathbb {CP}^1\times D$, where $D$ is a compact Riemann surface. Then $D=\{\infty\}\times D$ is a smooth divisor with trivial normal bundle. Fix a K\"ahler form $\omega_D$ on $D$ and also view it as a form on $\mathbb {CP}^1\times D$ via the pull-back of the obvious projection map. Note that up to a scaling $[\omega_D]\in c_1(\mathbb {CP}^1)$ in $H^{1,1}(\olsi{X})$. We can consider asymptotically cylindrical metrics on $X=\mathbb C\times D$ given by the K\"ahler forms 
	\begin{equation*}
		\omega=\omega_D+\ii\partial \pp \Phi(-\log |z|^2)=\omega_D+\frac{\varphi}{|z|^2}\ii d z\wedge d \olsi{z},
	\end{equation*}
	where $z$ denotes the coordinate function on $\mathbb C$ and $\varphi=\Phi''$ is a positive smooth function defined on $\mathbb R$ such that $\varphi(t)=e^{t}$ when $t$ is sufficiently negative and $\varphi(t)=1$ for $t$ sufficiently positive. Then one can easily check that $(X,\omega)$ satisfies \textit{Assumption 3} with $F(t)=t^2$. Let $\olsi{E}$ be an irreducible holomorphic vector bundle on $\mathbb {CP}^1\times D$ such that $\olsi{E}|_D$ is flat. Then by Theorem \ref{theorem for trivial normal bundle}, we know that 
	\begin{equation*}
		\text{$E$ admits an $\omega$-PHYM metric in $\mathcal P_{H_0}$ if and only if $\olsi{E}$ is $\left(c_1(D),c_1(\mathbb {CP}^1)\right)$-stable.}
	\end{equation*}
 \end{example}
Similar examples as in Example \ref{example of nonsplitting} show that the condition $\left(c_1(D),c_1(\mathbb {CP}^1)\right)$-stability is non-trivial. More specifically, let $D$ be a Riemann surface with genus $g\geq 1$ and $k\geq 2$ be an integer. Then similar argument as in Example \ref{example of nonsplitting} shows that there exists a non-splitting extension \begin{equation*}
	0\longrightarrow\mathcal O\longrightarrow E\longrightarrow p_1^*(\mathcal O_{\mathbb {P}^1}(-k))\longrightarrow0.
\end{equation*} 
whose restriction to $D$ splits. Then one can easily check that $E$ is irreducible and not $\left(c_1(D),c_1(\mathbb {CP}^1)\right)$-stable.

\subsection{Some problems for further study}\label{open problems}
Let $(X,\omega)$ satisfy the \textit{Assumption 1}. As illustrated by Theorem \ref{generalized tianyau} it is more natural to assume a stronger condition on the background K\"ahler metric $\omega$. More precisely, we assume that in (\ref{assumption on asymptotics}) the right hand side is replaced by $O(e^{-\delta_0 r^{\alpha_0}})$ for some $\delta_0, \alpha_0 >0$ and we also have the same bound for higher order derivatives. Under these assumptions and
motivated by the result of Hein \cite{hein2012} for solutions of complex Monge-Amp\`ere equations, we make the following conjecture.

\begin{conjecture}
	The solution $s$ obtained in Proposition \ref{existence result} decays exponentially, i.e $|\nabla^k s|=O(e^{-\delta r^{\alpha}})$ for some $\delta, \alpha >0$ and all $k\geq 0$.
\end{conjecture} 

Note that the key issue is to prove that $|s|$ decays exponentially, since all of the higher order estimates will follows form standard elliptic estimates. 

%Following Hein's argument in \cite{hein2012}-Section 2, we can show that $|s|$ is exponentially close to a constant, but it is not clear to the author that how to prove this constant has to be 0.

%If we have a confirmative answer for Conjecture 1, then naturally we have the following compactification problem. Let us focus on the complex dimension 2 case.
%Assume the above conjecture is true, then by Remark \ref{n=2 initial metric} and Proposition \ref{basic geometric property}-(4), we know that if $\olsi{E}|_D$ is polystable, then there exists a hermitian metric $H$ on $E$ such that
%\begin{equation*} \text{$\Lambda F_{H}^{\perp}=0$ and $|F_H|$ decay exponentially}.\end{equation*} Then we conjecture that every pair $(E,H)$ on $X$ satisfying these two conditions comes from our construction. More precisely, we have the following conjecture

%\begin{conjecture}
	 %Suppose we have a holomorphic vector bundle $E$ on $(X,\omega)$ with an $\omega$-PHYM metric $H$ satisfying that $|F_H|=O(e^{-\delta r^{\alpha}})$ for some $\delta,\alpha>0$. Then there exists a holomorphic vector bundle $\olsi{E}$ on $\olsi{X}$ whose restriction to $D$ is polystable and a smooth hermitian metric $\olsi{H}$ such that $E$ is isomorphic to $\olsi{E}|_X$ (as holomorphic vector bundles) and there is a smooth function $f$ on $X$ such that $H=\olsi{H}e^f$.\end{conjecture}

It is also an interesting problem to study the notion of $(\alpha,\beta)$-stability. From the definition, we have the following consequence:

\begin{proposition}\label{pair stability}
	Let $\alpha,\beta \in H^{1,1}(M)$ be two classes on a compact K\"ahler manifold $M$. Suppose $\alpha\in H^2(M,\mathbb Z)$ and $\alpha\wedge \beta=0$. Then a holomorphic vector bundle $E$ is $(\alpha,\beta)$-stable if and only if $E$ is $\alpha$-semistalbe and there exists an $\epsilon_0>0$ such that $\alpha+\epsilon \beta$-stable for all $0<\epsilon < \epsilon_0$.
\end{proposition}

It is natural to consider the following problem. Let $\olsi{E}$ be a $(c_1(D),[\omega_0])$-stable holomorphic vector bundle on $\olsi{X}$. Then by Proposition \ref{pair stability}, we know that $\olsi{E}$ is $[\omega_0]+\epsilon^{-1}c_1(D)$-stable for $\epsilon$ positive and sufficiently small. Donaldson-Uhlenbeck-Yau theorem says that for every K\"ahler form $\omega$ in $[\omega_0]+\epsilon^{-1}c_1(D)$, there exists an $\omega$-Hermitian-Yang-Mills metric on $\olsi{E}$. In our setting, it is natural to consider the following K\"ahler forms
\begin{equation}
	\omega_{\epsilon}=\omega_0+\epsilon^{-1}\theta_{\epsilon}+dd^c(\chi_{\epsilon} \varphi) \in [\omega_0]+\epsilon^{-1}c_1(D),
\end{equation}
where $\theta_{\epsilon}$ is a closed nonnegative (1,1)-form in $c_1(D)$, supported and positive in an $\epsilon$-neighborhood of $D$ and $\chi_{\epsilon}$ is a cut-off function which equals 1 outside an $\epsilon$-neighborhood of $D$ and $0$ in a smaller neighborhood of $D$. The K\"ahler forms are chosen such that 
\begin{equation*}
	\omega_{\epsilon}\longrightarrow \omega=\omega_0+dd^c\varphi \text{   in $C^{\infty}_{loc}(X)$}.
\end{equation*}

\begin{conjecture}
	Suppose $\olsi{E}|_D$ is $c_1(N_D)$-polystable and $(c_1(D),[\omega_0])$-stable. Let $H_{\epsilon}$ denote the Hermitian-Yang-Mills metric on $\olsi{E}$ with respect to the K\"ahler form $\omega_{\epsilon}$. Then there exists smooth functions $f_{\epsilon}$ on $X$ such that 
	\begin{equation*}
		H_{\epsilon}e^{f_{\epsilon}}\longrightarrow H \text{ in $C^{\infty}_{loc}(X)$},
	\end{equation*}
	where $H$ is the hermitian metric constructed in Theorem \ref{main theorem}.
\end{conjecture}

\bibliographystyle{plain}
\bibliography{ref.bib}

\end{document}